\documentclass{article}

\usepackage{amssymb,latexsym,amsmath,extarrows, mathrsfs, amsthm,color,amsrefs,dsfont}
\usepackage{graphicx} 
\usepackage{subcaption} 

\usepackage{geometry}
\usepackage[colorlinks=true, pdfstartview=FitV, linkcolor=black,citecolor=black, urlcolor=black]{hyperref}

\newtheorem{thm}{Theorem}[section]
\newtheorem{lem}{Lemma}[section]
\newtheorem{cor}[lem]{Corollary}
\newtheorem{prop}[lem]{Proposition}

\newtheorem{rem}{Remark}

\renewcommand{\Im}{\mathrm{Im}\,}

\definecolor{deepgreen}{cmyk}{1,0,1,0.5}

\def\calL{\mathcal{L}}

\def\R{\mathbb{R}}

\def\les{\lesssim}

\def\p{\partial}

\def\les{\lesssim}

\def\ol{\overline}
\def\calF{\mathcal{F}}
\def\tcalF{\widetilde \calF}

\def\one{\mathds{1}}

\def\nn{\nonumber}
\def\supp{\mathrm{supp}}

\def\eps{\varepsilon}

\newcommand{\norm}[1]{\left\lVert #1 \right\rVert}
\newcommand{\jap}[1]{\left\langle #1 \right\rangle}
\newcommand{\abs}[1]{\left\lvert #1 \right\rvert}

\usepackage{mathtools}
\mathtoolsset{showonlyrefs}
\numberwithin{equation}{section}

\title{The cubic NLS on the line with an inverse square potential
}
\author{J.\ Krieger\thanks{The first author thanks Yale University for its hospitality. He was partially supported through SNF grant 225701.}, W.\ Schlag\thanks{The second author thanks EPFL and Universit\"at Z\"urich  for their hospitality. He was partially supported by the NSF through grant DMS-2350356.},  K.\ Widmayer\thanks{The third author gratefully acknowledges support of the SNSF through grant PCEFP2\_203059.}}
\date{}

\newcommand{\Addresses}{{
  \bigskip
  \footnotesize

  J.~Krieger, \textsc{EPFL SB MATH PDE, B\^atiment MA, Station 8, CH-1015 Lausanne, Switzerland }\par\nopagebreak
  \textit{E-mail address:}  \texttt{joachim.krieger@epfl.ch}

  \medskip

  W.~Schlag, \textsc{Department of Mathematics,
    Yale, 219 Prospect Avenue, New Haven, CT 06510, USA}\par\nopagebreak
  \textit{E-mail address:}  \texttt{wilhelm.schlag@yale.edu}

  \medskip

  K.~Widmayer, \textsc{Faculty of Mathematics, University of Vienna, Oskar-Morgenstern-Platz 1, 1090 Vienna, Austria, \& Institute of Mathematics, University of Zurich, Winterthurerstrasse 190, 8057 Zurich, Switzerland}\par\nopagebreak
  \textit{E-mail address:}  \texttt{klaus.widmayer@univie.ac.at} \& \texttt{klaus.widmayer@math.uzh.ch}

}}

\begin{document}

\maketitle

\begin{abstract}
    We establish modified scattering for solutions of the cubic NLS on the line with a repulsive inverse square potential and small localized data. The method is based on a comparison between the free and distorted Galilei vector fields and a wave packet transform. 
\end{abstract}

\tableofcontents
\section{Introduction}

The purpose of this note is to study the long-term dynamics of solutions of the cubic NLS with a critical, repulsive potential
\begin{equation}
    \label{eq:mainPDE}
    i\p_t u + \calL u =\mu |u|^2 u,\qquad \calL = -\p_x^2 + \frac{2}{1+x^2} ,\quad \mu =\pm1,
\end{equation}
for small data which are sufficiently smooth and decaying. More precisely, we prove this

\begin{thm}
    \label{thm:main}
    The PDE~\eqref{eq:mainPDE} with small data in $\langle x\rangle^{-1}H^1(\R)$ has a unique global solution that decays like $t^{-\frac12}$ in $L^\infty(\R)$. Moreover, the solution exhibits modified scattering in $L^\infty(\R)\cap L^2(\R)$, in the sense that there exist a unique $u_\infty\in L^2(\R)\cap L^\infty(\R)$ and $\Phi_\infty\in L^\infty(\R)$ such that as $t\to\infty$
\begin{equation}\nn
  u(x,t)=t^{-\frac12}e^{\left(-i\frac{x^2}{4t}-i\mu \abs{u_\infty\left(\frac{x}{t}\right)}^2\log(t)-i\mu\Phi_\infty\left(\frac{x}{t}\right)\right)}u_\infty\left(\frac{x}{t}\right)+R[u](x,t),  
\end{equation}
with $R[u]\in L^2(\R)\cap L^\infty(\R)$ satisfying
\begin{equation}
    \norm{R[u](t)}_{L^\infty}\ll t^{-\frac{3}{5}},\qquad \norm{R[u](t)}_{L^2}\ll  t^{-\frac{1}{5}}.
\end{equation}
\end{thm}

See Theorem~\ref{thm:main2}  for a more detailed statement. 
Our methods apply more generally, i.e., to positive potentials $V(x)\sim c_{\pm} x^{-2}$ as $x\to\pm\infty$. But for simplicity, we only consider $c_\pm =2$ since this choice is explicit and seems to minimize the technicalities. 
Modified scattering for the cubic NLS on the line without a potential has been studied via perturbative methods, which do not rely on  complete integrability, by Ozawa~\cite{Ozawa}, Hayashi-Naumkin~\cite{HayNau}, Lindblad-Soffer~\cite{LindSof}, Kato-Pusateri~\cite{KatPus} and Ifrim-Tataru~\cite{IT}.
Generic potentials were added by Naumkin~\cite{Naum}, Germain-Pusateri-Rousset~\cite{GPR}, and Chen-Pusateri~\cite{ChenPus1}. For non-generic potentials see for example Chen-Pusateri~\cite{ChenPus2} and the references cited there.  Kawamoto and Mizutani~\cite{Kawamoto_2025} investigated the existence of modified wave operators for a NLS equation with a long-range potential.

\smallskip

We were motivated by  Gavin Stewart's recent paper~\cite{Gavin}, which deals with this PDE provided $\calL=-\p_x^2 +V$ with $\langle x\rangle^\gamma V\in L^1(\R)$ and $\gamma>2$. In addition, Stewart requires the distorted Fourier transform to vanish at zero energy (which happens in particular for the generic case). As remarked in~\cite{Gavin}, it is a delicate matter to lower the decay requirement on the potential. 

\smallskip

Our potential fails Stewart's  decay condition by a lot, and is  critical in the sense that it exhibits the same scaling as the Laplacian, to leading order. 
It is a basic fact of potential scattering on the line, see Deift-Trubowitz~\cite{DeiTru}, that the Jost solutions are continuous at energy zero only if $\langle x\rangle^\gamma V\in L^1(\R)$ and $\gamma>1$. Our potential precisely  fails this condition. The reason we are nevertheless able to treat~\eqref{eq:mainPDE} is the fact that we can trade decay for repulsivity. In particular, the transmission coefficient $T(\xi)$ vanishes to second order as $\xi\to0$. The reflection coefficient $R(\xi)\to 1$ as $\xi\to0$ in contrast to the case of generic $V$  with  $\langle x\rangle V\in L^1(\R)$ for which $R(0)=-1$. Smooth repulsive potentials of the inverse square type arise naturally in a number of settings, such as in the analysis of wave equations on surfaces of revolution, cf.~\cite{SSSI} and~\cite{SSSII}. They manifest themselves in any setting which admits some form of polar coordinate decomposition such as for waves on a black hole background, see for example Donninger, Soffer and the second author's work on Price's law~\cite{DSS}.

We rely on a modification of the Ifrim-Tataru~\cite{IT} wave packet technique. This stands in contrast to another well-known implementation of the space-time resonance approach to modified scattering based on the {\em nonlinear spectral distribution function}, see~\cites{ChenPus1, ChenPus2}. Similarly to~\cite{Gavin}, the objective is to compare the Galilean vector field $J_0(t)=x-2it\p_x$ with its distorted version. The latter is defined by the action of $i\p_\xi$ on the distorted Fourier transform of the profile $f(t)=e^{-it\calL}u(t)$. The advantage of passing from the distorted Galilei field to the free one lies with its algebraic properties such as a well-known product rule, viz.
\[
J_0(s)\big(|u|^2 u(s)\big) = 2J_0(s)u(s)\,|u(s)|^2 - u^2(s)\overline{J_0(s) u(s)}
\]
see~\eqref{eq:nullform} and~\cites{KS,KatPus,IT}, as well as~\eqref{eq:KSmfld} which played a role in~\cite{KS}. 

\smallskip 

The paper is organized as follows. Section~\ref{sec:spec} develops the scattering theory of the inverse square potential. Although there is overlap with \cite{SSSI} and~\cite{SSSII}, we need finer asymptotic information on the Jost solutions. This is done via a careful perturbative expansion around the leading-order Bessel functions. At zero energy $\xi=0$ a logarithmic singularity appears in the higher-order terms, see~\eqref{eq:log}. This precludes relying on a standard 
Calderon-Vaillancourt theorem for the $L^2$ bounds, and we prove the version that is needed here in Lemma~\ref{lem:CalVal}. As in~\cite{Gavin}, a key property in our analysis is played by the faster local pointwise decay of $e^{it\calL}$ on an interval $|x|\le t^\gamma$ with $\frac12\le \gamma<1$.  This is a manifestation of the repulsivity of the potential, which leads to faster local decay. 

\smallskip

   Sections~\ref{sec:J0JV} and~\ref{sec:propN} are the core of the paper and differ significantly from~\cite{Gavin}. The former compares free and distorted Galilei fields and contains a comparison between the standard and distorted Fourier transforms. The latter then shows that we can bootstrap the distorted Galilei field $J_V(s)$  through the cubic term -- see Proposition~\ref{prop:cubic}. This is done via a microlocal analysis isolating various regimes in phase space. On the physical side, we separate the inner region $|x|\le s^{\gamma}$ from the outer one $|x|>s^{\gamma}$  where $\gamma=\frac12+\delta$ for some small $\delta>0$. The inner one is comparatively easier due to the faster local decay mentioned above. 
The outer region requires distinguishing frequencies $|\xi|<s^{-\frac12}$ from $|\xi|\ge s^{-\frac12}$, see Lemmas~\ref{lem:xlargexismall}--\ref{lem:J_Vnonlinear2}. 

\smallskip

The final two sections then implement the wave packets from~\cite{IT} and the bootstrap analysis leading to Theorem~\ref{thm:main}.  Here we follow Stewart closely by noting that the wave packet analysis is only needed in the outer region, again thanks to the faster local decay.

\section{The spectral theory of $\calL$}
\label{sec:spec}

\subsection{Jost solutions and the scattering matrix}

We would like to find the Jost solutions $f_{\pm}(x,\xi)$ satisfying
\begin{equation}\label{eq:Jost_def}
\calL f_\pm(\cdot,\xi) = \xi^2 f_\pm(\cdot,\xi), \quad f_\pm(x,\xi) \sim e^{\pm ix\xi} \quad x\to\pm\infty.
\end{equation}
Note that $f_+(x,\xi)=\overline{f_+(x,-\xi)}$ since $x,\xi$ are real-valued, so it suffices to consider $\xi>0$. 
The turning point is given by $\frac{2}{1+x_*^2}=\xi^2$, or $x_*=\sqrt{2/\xi^{2}-1}$ provided $\xi^2<2$. The following lemma describes $f_+(x,\xi)$ for small $\xi>0$ in a region that extends through the turning point. We refer the reader to~\cite{SSSII} for a similar analysis of critical potentials (see the case $\nu=\frac32$ in \cite[Def.\ 3.1]{SSSII}). 

\begin{lem}
    \label{lem:Jost1}
    For $0<\xi\le 1\leq x$ there exists $\rho_1\in C^\infty$ such that 
    \begin{equation}   
        \label{eq:Jost1}
        f_+(x,\xi) = h_+(x\xi)(1+\xi^2 \rho_1(x\xi,\xi)),\quad h_+(u)=e^{iu}\left(1+ \frac{i}{u}\right),
    \end{equation}
   satisfies $\calL f_+=\xi^2 f_+$ with $f_+(x,\xi)\sim e^{ix\xi}$ as $x\to\infty$. More precisely, we have that for all $u\ge \xi$
   \begin{equation}\label{eq:rho1_1stbd}
    |(u\p_u)^a(\xi\p_\xi)^b\rho_1(u,\xi)|\les_{a,b} (1+u^{-2})(1+u)^{-3},\qquad a,b\in\mathbb{N}_0.
   \end{equation}
\end{lem}
We remark that since $V(x)$ is even there holds that 
\begin{equation}\label{eq:x-symm_fpm}
    f_-(x,\xi)=f_+(-x,\xi),
\end{equation}
and Lemma \ref{lem:Jost1} thus implies the existence of $f_-$ with corresponding asymptotic behavior as $x\to -\infty$.
\begin{proof}
With the ansatz $f_+(x,\xi)=g(x\xi,\xi)$ we obtain from \eqref{eq:Jost_def} that
\begin{equation}
    \label{eq:Hxi}
    (H_\xi g)(u)=g(u,\xi),\quad H_\xi=\left(-\p_u^2 + \frac{2}{\xi^2+u^2}\right),\quad g(u,\xi)\sim e^{iu} \quad u=x\xi\to\infty.
\end{equation}
We observe that the comparison operator 
\[
H_0 :=  -\p_u^2  + \frac{2}{u^2}
\]
admits explicit solutions $h_\pm(u)=e^{\pm iu}(1\pm i/u)$ for $u>0$. We therefore can solve~\eqref{eq:Hxi} in the form 
\begin{align}
    \label{eq:ghVolt}
    g(u,\xi) &= h_+(u) + 2\xi^2\int_u^\infty G(u,v) \frac{g(v,\xi)}{v^2(v^2+\xi^2)}\, dv, \\
    G(u,v) &= (2i)^{-1}[h_+(u)h_-(v)-h_+(v)h_-(u)] = \Im[h_+(u)h_-(v)] \nonumber \\
    &= \sin(u-v) \Big(1+\frac{1}{uv}\Big) + \cos(u-v) \Big( \frac{1}{u}-\frac{1}{v}\Big). \nonumber
 \end{align}
Writing $g(u,\xi)=h_+(u)\rho(u,\xi)$ yields
\[
\rho(u,\xi) = 1 + 2\xi^2\int_u^\infty \Im[h_+(u)\overline{h_+(v)}]\frac{h_+(v)}{h_+(u)} \frac{\rho(v,\xi)}{v^2(v^2+\xi^2)}\, dv  = 1 + \xi^2 \rho_1(u,\xi).
\]
Combining this with the explicit form of $G(u,v)$,
\begin{equation}\label{eq:rho1}
\begin{aligned}
    \rho_1(u,\xi) &= \int_u^\infty \Big[\big(1+e^{-2i(u-v)}\big)(1/u-1/v) -i \big(1-e^{-2i(u-v)}\big)(1+1/(uv))\Big] \\
    &\qquad\qquad \cdot\frac{1+i/v}{1+i/u} \frac{\rho(v,\xi)}{v^2(v^2+\xi^2)} \, dv \\
    &=\rho_{10}(u,\xi) + \int_u^\infty K(u,v;\xi) \rho_1(v,\xi)\, dv,
\end{aligned}
\end{equation}
where 
\begin{equation}\label{eq:rho10etc}
\begin{aligned}
 K(u,v;\xi)&=K_0(u,v)\frac{v+i}{u+i} \;\frac{\xi^2}{v^4(v^2+\xi^2)},\\
 K_0(u,v)&=\big(1+e^{-2i(u-v)}\big)(v-u) -i \big(1-e^{-2i(u-v)}\big)(1+uv) ,\\
 \rho_{10}(u,\xi)&=\int_u^\infty \xi^{-2}K(u,v;\xi)dv.
\end{aligned}
\end{equation}
By inspection, for any $u\geq 1$
\begin{align*}
    |\rho_{10}(u,\xi)| &\les \int_u^\infty \Big[v-u + 1+uv\Big]\frac{v+1}{u+1} \;\frac{dv}{v^4(v^2+\xi^2)} \le C_0 u^{-3}, 
\end{align*}
while for $\xi\le u\le 1$ we expand uniformly in $u\le v\les 1$
\begin{equation}\label{eq:cubic}
\begin{aligned}
K_0(u,v)& = 2(u-v) uv -2i(u-v)^2uv+ O((u-v)^3) = O(v^3)
\end{aligned}
\end{equation}
whence 
\begin{align*}
    |\rho_{10}(u,\xi)| &\les \int_u^1   \;\frac{dv}{v(v^2+\xi^2)}
    +\int_1^\infty \Big[v-u + 1+uv\Big]\frac{v+1}{u+1} \;\frac{dv}{v^6} 
     \le C_0 u^{-2},
\end{align*}
so that altogether
\begin{equation}
    \label{eq:weak_bd}
    |\rho_{10}(u,\xi)| \le C_0(1+u^{-2}) (1+u)^{-3}. 
\end{equation}
The Volterra equation \eqref{eq:rho1} then implies the bound \eqref{eq:rho1_1stbd} on $\rho_1$: by Lemma~2.4 in~\cite{SSSI}\footnote{applied in weighted $L^\infty$ with the weight $(1+u^{-2}) (1+u)^{-3}$ given by \eqref{eq:weak_bd}}, it suffices to verify that the bound is compatible with Volterra iteration. This holds since for all $u\ge1$
\begin{align*}
    u^3 \int_u^\infty \sup_{u\le w\le v}|K(w,v;\xi)| v^{-3}\, dv &\les \xi^2 u^{3}\int_u^\infty  \Big[v-u + 1+uv\Big]\frac{v+1}{u+1} \,\frac{1}{v^7(v^2+\xi^2)} \, dv \\
&\les \xi^2 u^3 \int_u^\infty v^{-7}\, dv      \les \xi^2 u^{-3} \les \xi^2 
\end{align*}
 uniformly in small $\xi>0$, while for $\xi\le u\le 1$ we have
\begin{align*}
    u^2 \int_u^\infty \sup_{u\le w\le v}|K(w,v;\xi)|(1+v^{-2}) (1+v)^{-3}\, dv&\les \xi^2 u^2 \int_u^1   \,\frac{1}{v(v^2+\xi^2)} v^{-2} \, dv      + \xi^2 u^2 \int_1^\infty   v^{-7} \, dv \\
   &\les \xi^2 u^{-2}\les 1.
\end{align*}
To bound the derivatives we compute that
\begin{equation}\label{eq:dudvK0}
    \p_uK_0(u,v)+\p_vK_0(u,v)=-i(u+v)(1-e^{-2i(u-v)})=:u^{-1}\widetilde{K}_0(u,v),
\end{equation}
so that with $K_0(u,u)=0$ an integration by parts yields
\begin{equation}
\begin{aligned}
  u\p_u   \rho_{10}(u,\xi) &=\int_u^\infty \left(-u\p_v K_0(u,v)+\widetilde{K}_0(u,v)-K_0(u,v)\frac{u}{u+i}\right) \frac{v+i}{u+i} \frac{dv}{v^4(v^2+\xi^2)}\\
  &=\int_u^\infty\widetilde{K}(u,v;\xi) dv,\\
  \widetilde{K}(u,v;\xi)&=K_0(u,v)\left(u\p_v\left(\frac{v+i}{u+i} \frac{1}{v^4(v^2+\xi^2)}\right)-\frac{u}{u+i}\frac{1}{v^4(v^2+\xi^2)}\right)+\widetilde{K}_0(u,v)\frac{1}{v^4(v^2+\xi^2)}.
\end{aligned}    
\end{equation}
Similarly, from \eqref{eq:rho1} we have
\begin{equation}
\begin{aligned}
 u\p_u\rho_1(u,\xi)&=u\p_u \rho_{10}(u,\xi) +\int_u^\infty u\p_u K(u,v;\xi)\rho_1(v,\xi)dv\\
 &=\rho_{10}^{(1)}(u,\xi)+ \int_u^\infty \frac{u}{v} K(u,v;\xi)v\p_v\rho_1(v,\xi)dv,
\end{aligned} 
\end{equation}
where
\begin{equation}
  \rho_{10}^{(1)}(u,\xi)= u\p_u\rho_{10}(u,\xi)+\xi^2\int_u^\infty \widetilde{K}(u,v;\xi)\rho_1(v,\xi)dv.
\end{equation}
With the same arguments as above, using also that $\widetilde{K}_0(u,v)=O(v^3)$ when $u\leq v\leq 1$, one obtains that
\begin{equation}
  \abs{\rho_{10}^{(1)}(u,\xi)}\lesssim  (1+u^{-2})(1+u)^{-3},
\end{equation}
and as above Volterra iteration gives the corresponding bound for $u\p_u\rho_1(u,\xi)$. Repeating this procedure gives the bound \eqref{eq:rho1_1stbd} for $a\in\mathbb{N}_0$ and $b=0$. To incorporate also $\xi\p_\xi$ derivatives it suffices to notice that $\xi\p_\xi \xi^2 = 2\xi^2$ and $\xi\p_\xi (v^2+\xi^2)^{-1} = - 2\xi^2 (v^2+\xi^2)^{-2}$, so that in the above arguments the kernel bounds are unchanged.
\end{proof}

Next we solve from $x=0$ towards the turning point, and then we address the connection problem. 

\begin{lem}
    \label{lem:Jost2}
    There exists $\delta>0$ such that uniformly in $0<\xi\le 1$ the following holds: For $x\in[0,\delta \xi^{-1}]$ the ODE $\calL \varphi=\xi^2\varphi$ admits a real-valued, smooth fundamental system $\{\varphi_1,\varphi_2\}$ of the form 
    \begin{equation}\label{eq:phi12}
    \begin{aligned}
        \varphi_1(x,\xi) &= (1+x^2) (1+\xi^2 \tau(x,\xi)),\\
        \varphi_2(x,\xi)  &=  \varphi_1(x,\xi) \int_x^{\delta\xi^{-1}} \varphi_1(y,\xi)^{-2}\, dy.
    \end{aligned}
    \end{equation}
    Their Wronskian satisfies $W[\varphi_1,\varphi_2]=1$, and we have the bounds
    \begin{equation}\label{eq:tau-phi2-bds}
    \begin{aligned}
     \abs{(\jap{x}\p_x)^a(\xi\p_\xi)^b\tau(x,\xi)}\les_{a,b}\jap{x}^2,\quad \abs{(\jap{x}\p_x)^a(\xi\p_\xi)^b\varphi_2(x,\xi)}\les_{a,b}\jap{x}^{-1},\quad a,b\in\mathbb{N}_0.
    \end{aligned}
    \end{equation}
\end{lem}
We remark that by symmetry of the equation, for $x\in[-\delta\xi^{-1},0]$ the analogous claim holds for the functions $\varphi_j(-x,\xi)$, $j=1,2$.
\begin{proof}
    First, note that $\calL f=0$ has a fundamental system $\{f_1,f_2\}$ of the form \begin{align*}
    f_1(x) &=1+x^2,\qquad f_2(x) = (\pi/2-\arctan x)(1+x^2)-x.
\end{align*}
(Note that $f_2(x)=O(x^{-1})$ as $x\to\infty$.) We find a solution $\varphi_1$ of $\calL \varphi_1=\xi^2\varphi_1$ in the form 
\begin{align*}
    \varphi_1(x,\xi) & = f_1(x) + \frac{\xi^2}{2} \int_0^x [f_1(x)f_2(y)-f_1(y)f_2(x)]\varphi_1(y,\xi)\, dy= f_1(x)(1 + \xi^2\tau(x,\xi)),
\end{align*}
where $\tau(x,\xi)$ satisfies for $0\le x\le\xi^{-1}$
\begin{align*}
    \tau(x,\xi) &= \frac12 \int_0^x [f_2(y)-f_1(y)f_2(x)/f_1(x)]f_1(y)  (1+\xi^2 \tau(y,\xi))\, dy \\
    &= \tau_0(x) + \frac{ \xi^2}{2} \int_0^x G(x,y)  \tau(y,\xi)\, dy,\\
    G(x,y)&=(f_2(y)-f_1(y)f_2(x)/f_1(x))f_1(y).
\end{align*}
Here
\begin{equation}\label{eq:log}
 \tau_0(x)=\frac12 \int_0^x G(x,y)dy=\frac{1}{30} \left(3x^2+8+4\log \left(x^2+1\right)-\frac{8}{x^2+1}\right)=\frac{x^2}{2}+O(x^4),
\end{equation}
so in particular 
\begin{equation}\label{eq:tau0_bd}
\abs{(\jap{x}\p_x)^a(\xi\p_\xi)^b\tau_0(x)}\les_{a,b}\jap{x}^2,\qquad a,b\in\mathbb{N}_0.
\end{equation}
For $\sigma(x,\xi):=\jap{x}^{-2}\tau(x,\xi)$ we then obtain
\begin{align*}
\sigma(x,\xi) & = \sigma_0(x) +  \int_0^x \frac{ \xi^2}{2}\jap{x}^{-2}G(x,y)\jap{y}^{2}\sigma(y,\xi)\, dy,
\end{align*}
with $\abs{(\jap{x}\p_x)^a(\xi\p_\xi)^b\sigma_0(x)}\les_{a,b}1$ uniformly in $x\geq 0$ by \eqref{eq:tau0_bd}. For $0\leq y\leq x\leq \xi^{-1}$ we also have 
\begin{equation*}
|\xi^2\jap{x}^{-2}G(x,y)| \les \xi^2\langle x\rangle^{-2} (1+y)^3\les \xi^2 \langle y\rangle,
\end{equation*}
so that by Volterra iteration (cf.~\cite{SSSI}*{Lemma 2.4}),  we conclude that $\sigma$ is uniformly bounded in the specified range of parameters, which then implies the desired estimate on~$\tau$. The corresponding bounds for derivatives $\xi\p_\xi$ follow directly using the explicit dependence on $\xi$, whereas the bounds for derivatives in $\jap{x}\p_x$ follow from the fact that for $a,b\in\mathbb{N}_0$
\begin{equation}
\abs{(\jap{x}\p_x)^a\left(\jap{x}^{-2}G(x,y)\right)}\les_{a,b}\jap{y}.
\end{equation}

Another solution to $\calL\varphi=\xi^2\varphi$ is given by 
\begin{align*}
    \varphi_2(x,\xi) & =  \varphi_1(x,\xi)\int_x^{\delta\xi^{-1}}  \varphi_1(y,\xi)^{-2}\, dy \\
    & = f_1(x)(1 + \xi^2\tau(x,\xi))\int_x^{\delta\xi^{-1}} \big[ f_1(y)(1 + \xi^2\tau(y,\xi))\big]^{-2}\, dy,
\end{align*}
where $\delta>0$ is chosen so small that $\xi^2|\tau(y,\xi)|\le\frac12$ for $y\in[0,\delta\xi^{-1}]$.
By inspection, $\varphi_2(x,\xi)=O(\langle x\rangle^{-1})$ in the same range of~$x$, and it is immediate that this bound is stable under powers of~$\jap{x}\p_x$. As for derivatives in $\xi\p_\xi$, thanks to the corresponding bounds on $\tau$ and thus also $\varphi_1$, it suffices to observe that for $0\leq x\leq \delta\xi^{-1}$
\begin{equation}
    \abs{(\xi\p_\xi)^{a+1}\int_x^{\delta\xi^{-1}}  \varphi_1(y,\xi)^{-2}\, dy}=\abs{(\xi\p_\xi)^{a}\left(\delta \xi^{-1}\varphi_1(\delta\xi^{-1},\xi)^{-2}\right)}\les_a\jap{x}^{-3},
\end{equation}
as claimed. 
\end{proof}
For future reference we record from the proof that by construction
\begin{equation}\label{eq:phi12_info}
\begin{aligned}
  \varphi_1(0,\xi)&=1,\qquad 4\int_0^{\delta\xi^{-1}}(1+y^2)^{-2}dy\geq\varphi_2(0,\xi)\geq\frac{1}{4}\int_0^{\delta\xi^{-1}}(1+y^2)^{-2}dy,\\
  \p_x\varphi_1(0,\xi)&=0,\qquad \p_x\varphi_2(0,\xi)=-1.
\end{aligned}    
\end{equation}

\begin{rem}
 For ease of notation we will henceforth refer to functions satisfying iterated derivative bounds as in \eqref{eq:rho1_1stbd} or \eqref{eq:tau-phi2-bds} as functions exhibiting \emph{symbol-type behavior} or being of \emph{symbol type}. Equivalently, we will state bounds for a function and say they are \emph{stable under arbitrary powers} of a zero-homogeneous derivative (e.g., $u\p_u$ or $\xi\p_\xi$) if the function exhibits symbol-type behavior.
\end{rem}

We comment briefly on the well-known case of large $\xi$, again restricting to $\xi>0$ and asymptotics for $x\to\infty$ by symmetry (see also \eqref{eq:x-symm_fpm}).
\begin{lem}\label{lem:Jost3}
 Let $\xi>1$. There exists $m_+\in C^\infty$ such that   
 \begin{equation}\label{eq:m_+def}
     f_+(x,\xi)=e^{ix\xi}m_+(x,\xi)
 \end{equation}
 satisfies $\calL f_+=\xi^2f_+$, with 
 \begin{equation}
   m_+(x,\xi)-1=O(\xi^{-1}\jap{x}^{-1}),\qquad x>0,\quad \xi\to\infty,
 \end{equation}
 this bound being stable under arbitrary powers of $\jap{x}\p_x$ and $\xi\p_\xi$.
\end{lem}
In view of \eqref{eq:x-symm_fpm} and the fact that $x,\xi$ are real-valued, we have for $x,\xi>0$
\begin{equation}
    m_-(-x,\xi)=m_+(x,\xi),\quad \overline{m_+(x,\xi)}=m(x,-\xi).
\end{equation}
\begin{proof}
 The arguments are well-known \cite{DeiTru}, so we only sketch the main points. The ansatz \eqref{eq:m_+def} implies that $\p_x^2m_++2i\xi\p_xm_+=V(x)m_+$, $V(x)=\frac{2}{1+x^2}$, and with the requirement that $m_+(x,\xi)\to 1$ as $x\to\infty$ we obtain $m_+$ by solving the Volterra equation
 \begin{equation}
  m_+(x,\xi) = 1 + \int_x^\infty K(x,y;\xi) m_+(y,\xi)\, dy,\qquad K(x,y;\xi):=\frac{1-e^{2i\xi(y-x)}}{2i\xi} \frac{2}{1+y^2}.
 \end{equation}
 From this and Volterra iteration, the bound for $m_+(x,\xi)-1$ immediately follows. Moreover, upon differentiation and integration by parts there holds that
 \begin{equation*}
 \begin{aligned}
    \p_xm_+(x,\xi)&=\int_x^\infty e^{2i\xi(y-x)}\frac{2}{1+y^2}\, m_+(y,\xi)\, dy\\
    &=-\frac{1}{2i\xi}\frac{2}{1+x^2}m_+(x,\xi)+\int_x^\infty \frac{e^{2i\xi(y-x)}}{2i\xi}\frac{4y}{(1+y^2)^2}\, m_+(y,\xi)\, dy\\
    &\qquad -\int_x^\infty \frac{e^{2i\xi(y-x)}}{2i\xi} \frac{2}{1+y^2}\, \p_xm_+(y,\xi)\, dy,
 \end{aligned}   
 \end{equation*}
 and it follows that $\abs{\p_xm_+(x,\xi)}=O(\xi^{-1}\jap{x}^{-2})$. Similarly,
 \begin{equation*}
 \begin{aligned}
    \xi\p_\xi m_+(x,\xi)&=-\int_x^\infty K(x,y;\xi) m_+(y,\xi)\, dy+\int_x^\infty(y-x)e^{2i\xi(y-x)}\frac{2}{1+y^2}\, m_+(y,\xi)\, dy\\
    &\qquad +\int_x^\infty K(x,y;\xi) \xi\p_\xi m_+(y,\xi)\, dy\\
    &=-\int_x^\infty K(x,y;\xi) m_+(y,\xi)\, dy-\int_x^\infty \frac{e^{2i\xi(y-x)}}{2i\xi}\frac{2}{1+y^2}\left(1-\frac{2y(y-x)}{1+y^2}\right)\, m_+(y,\xi)\, dy\\
    &\qquad +\int_x^\infty K(x,y;\xi) \xi\p_\xi m_+(y,\xi)\, dy,
 \end{aligned}   
 \end{equation*}
 which proves the bound for $\p_\xi m_+(x,\xi)$. Finally, the claim for higher order derivatives follows by iteration.
\end{proof}

Next, we will find the shape of  the distorted Fourier basis
\begin{equation}
\label{eq:dFT}
e(x,\xi) = \frac{1}{\sqrt{2\pi}}
\begin{cases}
T(\xi) f_+(x,\xi), & \xi\ge0, \\
T(-\xi) f_+(-x,-\xi), & \xi\le0. 
\end{cases}
\end{equation}
Here, 
\begin{equation}\label{eq:T_def}
    T(\xi)=-\frac{2i\xi}{W[f_+,f_-](\xi)}
\end{equation}
is the transmission coefficient and we have used that $f_-(x,\xi)=f_+(-x,\xi)$ by symmetry of the potential, see also \eqref{eq:x-symm_fpm}. In particular, there holds that
\begin{equation}\label{eq:exxi_symm}
    e(x,\xi)=e(-x,-\xi).
\end{equation}Another consequence of this symmetry is the fact that the reflection coefficients $R_\pm(\xi)$ defined via
\begin{equation}\label{eq:exxi}
 T(\xi)f_\pm(x,\xi) = f_\mp(x,-\xi) + R_\mp(\xi) f_\mp(x,\xi)
\end{equation}
satisfy
\begin{equation}\label{eq:Refl_def}
 R_+(\xi)=R_-(\xi)=:R(\xi)=\frac{W[\overline{f_-(\cdot,\xi)},f_+(\cdot,\xi)]}{W[f_+(\cdot,\xi),f_-(\cdot,\xi)]}.   
\end{equation}

We begin with the entries of the scattering matrix.  Note the contrast to generic potentials that decay like $\langle x\rangle^{-2-}$ where $R_\pm(0)=-1$.
\begin{lem}
    \label{lem:Wrons}
    As $\xi\to0+$, the transmission and reflection coefficients satisfy
    \begin{equation}\label{eq:TR_basics}
       T(\xi)= t_0 \,\xi^{3}(1+O(\xi^{3})),\quad R(\xi)=1 + O(\xi^{1/2}),
    \end{equation}
    where $t_0\in\mathbb{C}\setminus\{0\}$. Moreover, these bounds are stable under  differentiation by $\xi\p_\xi$. 
\end{lem}
\begin{proof}
In view of Lemma \ref{lem:Jost1}, for small $\xi>0$ and $x\geq 1$
\begin{equation*}
 f_+(x,\xi)= h_+(x\xi)\big(1+\xi^2\rho_1(x\xi,\xi)\big),\qquad h_+(u)=e^{iu}(1+i/u),   
\end{equation*}
 which we continue for $x\ge0$ in the form
 \begin{equation}
     \label{eq:fplus}
      f_+(x,\xi) = c_1(\xi)\varphi_1(x,\xi) + c_2(\xi)\varphi_2(x,\xi),
 \end{equation}
 with $\varphi_1,\varphi_2$ as in Lemma \ref{lem:Jost2}. Using that $W[\varphi_1,\varphi_2]=1$, one has 
 \begin{equation*}
   c_1(\xi) = W[f_+,\varphi_2](\xi) , \quad  c_2(\xi) = - W[f_+,\varphi_1] (\xi).  
 \end{equation*}
 On the one hand, since $\varphi_2(\delta\xi^{-1},\xi)=0$ and $\partial_x\varphi_2(\delta\xi^{-1},\xi)=-\varphi_1(\delta\xi^{-1},\xi)^{-1}$, we have
 \begin{equation}\label{eq:c1}
 \begin{aligned}
 c_1(\xi) &= W[f_+,\varphi_2](\xi) = -f_+(\delta \xi^{-1},\xi)\p_x\varphi_2(\delta\xi^{-1},\xi) = f_+(\delta\xi^{-1},\xi)/\varphi_1(\delta\xi^{-1},\xi)  \\
 &= h_+(\delta)\big(1+\xi^2\rho_1(\delta,\xi)\big) [ (1+\delta^2\xi^{-2})(1 + \xi^2\tau(\delta\xi^{-1},\xi))]^{-1} = O(\xi^2).
 \end{aligned}
 \end{equation}
On the other hand, as $\xi\to0+$
 \begin{align}\label{eq:c2}
 c_2(\xi) &= W[f_+,\varphi_1](\xi) = \p_xf_+(\delta\xi^{-1},\xi)\varphi_1(\delta\xi^{-1},\xi)-f_+(\delta\xi^{-1},\xi)\p_x\varphi_1(\delta\xi^{-1},\xi) =O(\xi^{-1}).
 \end{align}
A more careful calculation, evaluating \eqref{eq:c2_expanded} at $x=\xi^{1/2}$ shows that
 \begin{equation}
 \label{eq:c2weak}
 c_2(\xi) = -3i\xi^{-1}(1+O(\xi^{1/2})).
 \end{equation}
 Note also that from the explicit expressions \eqref{eq:c1} and \eqref{eq:c2} it follows that $c_1$, $c_2$ exhibit symbol behavior.
For the reflection coefficient we obtain from \eqref{eq:Refl_def} and the symmetry property \eqref{eq:x-symm_fpm} that
 \begin{equation}
 \begin{aligned}
     R(\xi) &= -\frac{W[f_-(\cdot,-\xi),f_+(\cdot,\xi)]}{W[f_-(\cdot,\xi),f_+(\cdot,\xi)]} =-\frac12\left( \frac{\p_xf_+(0,-\xi)}{\p_xf_+(0,\xi)} + \frac{f_+(0,-\xi)}{f_+(0,\xi)} \right) \label{eq:Refl}\\
     &= -\frac{\overline{c_2(\xi)}}{c_2(\xi)} (1+O(\xi^2)) = 1+O(\xi^{1/2}),
 \end{aligned}
\end{equation}
where we have used that by \eqref{eq:c1} and \eqref{eq:phi12_info} there holds that
\begin{equation}
\begin{aligned}
  f_+(0,\xi)&=c_1(\xi)+c_2(\xi)\varphi_2(0,\xi)=c_2(\xi)\varphi_2(0,\xi)+O(\xi^2),\qquad\p_xf_+(0,\xi)=-c_2(\xi).
\end{aligned}  
\end{equation}
 For the transmission coefficient we note that by symmetry \eqref{eq:x-symm_fpm} and \eqref{eq:phi12_info} one has
 \begin{equation*}
 W[f_+,f_-](\xi) = 2 \p_xf_+(0,\xi) f_+(0,\xi) =-2c_2(\xi)^2\varphi_2(0,\xi)+O(\xi)=O(\xi^{-2}).
 \end{equation*}
 From \eqref{eq:T_def} we thus obtain
 \begin{equation}\label{eq:Txi}
  T(\xi)=-\frac{2i\xi}{W[f_+,f_-](\xi)}=i\frac{\xi}{c_2(\xi)^2\varphi_2(0,\xi)}\left(1+O(\xi^3)\right),  
 \end{equation}
 which establishes \eqref{eq:TR_basics}.
 The claim about derivatives follows from the corresponding statements in the previous lemmas. 
\end{proof}

We will need the following representation of the distorted Fourier basis from~\eqref{eq:dFT}. Due to the symmetry \eqref{eq:exxi_symm} it suffices to consider $\xi>0$. Figure~\ref{fig:1} shows a numerically computed distorted basis function~$e(x,1/5)$.

\begin{prop}
    \label{prop:dFT}
    For all $\xi>0$
    \begin{equation}
    \label{eq:e split}
    e(x,\xi)= e^{ix\xi} a(x,\xi) + e^{-ix\xi} b(x,\xi)
    \end{equation}
    where $a,b$ are smooth functions and $b(x,\xi)=0$ for all $x\ge-\frac12\xi^{-1}$ and the following bounds hold: For $\xi\leq 1$,  
    \begin{align*}
    |a(x,\xi)|&\les \xi^3\Big(1+ \xi^{-1} \jap{x}^{-1}\Big) \one_{[x\ge0]} + \xi^2\jap{x}^2\one_{[-\frac12\xi^{-1}<x<0]}+\one_{[x\le-\frac12\xi^{-1}]},\\ 
    |b(x,\xi)|&\les  \one_{[x\le-\frac12\xi^{-1}]},
    \end{align*}
    whereas for $\xi\ge 1$,   
    \[
    a(x,\xi)=1+ O(\xi^{-1}\langle x\rangle^{-1}),\quad  b(x,\xi) = O(\xi^{-1})\one_{[x\le0]}.
    \]
    All stated bounds are stable under arbitrary differentiation by $\xi\p_\xi$ and $\langle x\rangle\p_x$. 
\end{prop}

\begin{figure}[ht]
    \centering
    \begin{subfigure}{0.48\textwidth}
        \centering
        \includegraphics[width=\linewidth]{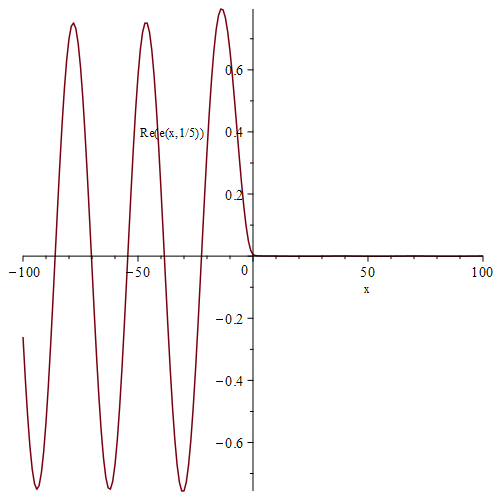}
    \end{subfigure}
    \hfill
    \begin{subfigure}{0.48\textwidth}
        \centering
        \includegraphics[width=\linewidth]{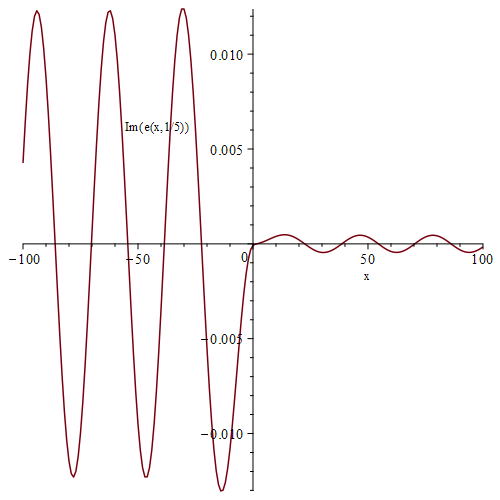}
    \end{subfigure}
    \caption{Real and imaginary parts of $e(x,1/5)$.  The  calculations were performed by Maple~\cite{maple}, numerically integrating the eigenvalue equation from $x=100=20\xi^{-1}$ to $x=-100$, using the leading asymptotics of $f_+(x,\xi)$ from Lemma~\ref{lem:Jost1} to find the initial conditions. To pass to $e(x,\xi)$, we calculate from~\eqref{eq:Txi} that  $T(\xi)= -\frac{4i}{9\pi}\xi^3 + \ldots$ to leading order. The discrepancy between the sizes of the real and imaginary parts of $e(x,\xi)$, which is clearly visible in terms of the scales of the $y$-axes in Figure~\ref{fig:1}, is due to $R(\xi)=-1+o(1)$ for small~$\xi$. The oscillations of the real part for $x>0$ are then too small in amplitude relative to the oscillations for $x<0$ to be resolved in the plot.  }
    \label{fig:1}
\end{figure}

In view of \eqref{eq:exxi_symm} we will henceforth extend the domain of $a,b$ to negative $\xi<0$ through
\begin{equation}\label{eq:ab-symm-def}
    a(x,\xi):=a(-x,-\xi),\qquad b(x,\xi):=b(-x,-\xi).
\end{equation}

\begin{proof} 
Let $\chi$ be smooth cutoff function on the line with $\chi(u)=1$ if $u\ge1$, and $\chi(u)=0$ if $u\le\frac12$. Then define
\begin{equation}\label{eq:def_ab}
\begin{aligned}
    a(x,\xi) &= (2\pi)^{-\frac12} e^{-ix\xi} (1-\chi(-x\xi)) T(\xi) f_+(x,\xi) +(2\pi)^{-\frac12} e^{-ix\xi} \chi(-x\xi)  f_-(x,-\xi),\\
    b(x,\xi)&= (2\pi)^{-\frac12} e^{ix\xi} \chi(-x\xi) R(\xi) f_-(x,\xi).
\end{aligned}
\end{equation}
One checks that \eqref{eq:e split} holds. We discuss next bounds in the case of small $0<\xi\lesssim 1$.

For $x\geq \delta\xi^{-1}$, Lemmas~\ref{lem:Jost1} and~\ref{lem:Wrons} imply that
\begin{equation}
       \label{eq:axxibig}
        a(x,\xi) = T(\xi)\Big( 1+\frac{i}{x\xi}\Big) (1+\xi^2\rho_1(x\xi,\xi)), \quad b(x,\xi)=0,
   \end{equation}
with $T(\xi)=O(\xi^3)$, so that $\abs{a(x,\xi)}\les \xi^3$. Since $a$ is defined as a product of terms exhibiting symbol behavior in $\xi\p_\xi$ and $u\p_u$, it does so as well in terms of $\xi\p_\xi$ and $x\p_x$.

For $0\leq x< \delta\xi^{-1}$,  we deduce from \eqref{eq:fplus} that
\begin{equation}\label{eq:axxismall}
   a(x,\xi) = (2\pi)^{-\frac12}e^{-ix\xi} T(\xi) \big( c_1(\xi)\varphi_1(x,\xi) + c_2(\xi)\varphi_2(x,\xi)\big),\qquad b(x,\xi)=0,
   \end{equation}
   which implies that $|a(x,\xi)|\les \xi^3\big(\xi^2\jap{x}^2+\xi^{-1}\jap{x}^{-1}\big)\les \xi^2 \jap{x}^{-1}$, as claimed. Again the symbol property and bounds follow from the corresponding bounds on the factors: Together with $\abs{(\jap{x}\p_x)^a e^{ix\xi}}\lesssim 1$ the symbol property and bounds for $a$ follow from the corresponding bounds on the factors in Lemmas \ref{lem:Jost2} and \ref{lem:Wrons}.

For $-\delta\xi^{-1}<x\leq0$ we expand  
\begin{equation}\label{eq:fplus_neg}
   f_+(x,\xi) = \widetilde c_1(\xi) \varphi_1(-x,\xi) + \widetilde c_2(\xi) \varphi_2(-x,\xi).
\end{equation}   
Here, computing $W[f_+(\cdot,\xi),\varphi_j(-\cdot,\xi)]$, $j=1,2$, and comparing at $x=0$ for the expansions \eqref{eq:fplus} and \eqref{eq:fplus_neg} and using \eqref{eq:phi12_info}, \eqref{eq:c2weak}, we have
\begin{align}
   \widetilde c_1(\xi) &= c_1(\xi) W[ \varphi_1(\cdot,\xi),\varphi_2(-\cdot,\xi)] + 
    c_2(\xi) W[ \varphi_2(\cdot,\xi),\varphi_2(-\cdot,\xi)] \nn \\
    &= - c_1(\xi)  \varphi_1(0,\xi)\varphi_2(0,\xi) - 2 c_2(\xi) \varphi_2(0,\xi)\varphi_2'(0,\xi) \nn \\
    & = - c_1(\xi) \varphi_2(0,\xi) + 2 c_2(\xi) = -6i\xi^{-1}(1+O(\xi^{1/2})) \label{eq:tilc1},
\end{align}
and similarly
   \begin{align}
   \widetilde c_2(\xi) &= -c_1(\xi) W[ \varphi_1(\cdot,\xi),\varphi_1(-\cdot,\xi)] - 
    c_2(\xi) W[ \varphi_1(\cdot,\xi),\varphi_2(-\cdot,\xi)]\nn \\
    & = -c_2(\xi)\p_x\varphi_2(0,\xi)=c_2(\xi) =-3i\xi^{-1}(1+O(\xi^{1/2})).  \label{eq:tilc2}
   \end{align}
Thus
\begin{equation}
  a(x,\xi) = (2\pi)^{-\frac12}e^{-ix\xi} T(\xi) \big(\widetilde c_1(\xi) \varphi_1(-x,\xi) + \widetilde c_2(\xi) \varphi_2(-x,\xi)\big),\qquad b(x,\xi)=0,  
\end{equation}
satisfies $\abs{a(x,\xi)}\les \xi^2\jap{x}^2$, with symbol bounds following as before.   

For $x<-\delta\xi^{-1}$ we use \eqref{eq:exxi} to write
\begin{equation}\label{eq:def_ab_neg}
\begin{aligned}
    a(x,\xi) &= (2\pi)^{-\frac12} e^{-ix\xi} (1-\chi(-x\xi)) R(\xi) f_-(x,\xi) +(2\pi)^{-\frac12} e^{-ix\xi} f_-(x,-\xi),\\
    b(x,\xi)&= (2\pi)^{-\frac12} e^{ix\xi} \chi(-x\xi) R(\xi) f_-(x,\xi).
\end{aligned}
\end{equation}
Since by \eqref{eq:x-symm_fpm} and Lemma \ref{lem:Jost1} there holds that
\begin{equation}
    f_-(x,\xi)=f_+(-x,\xi)=e^{-ix\xi}(1-\frac{i}{x\xi})(1+\xi^2\rho_1(-x\xi,\xi)),\quad f_-(x,-\xi)=\overline{f_-(x,\xi)},
\end{equation}
the symbol behavior and bounds $\abs{a(x,\xi)}\les 1$ and $\abs{b(x,\xi)}\les 1$ follow as in the case $x\geq \delta\xi^{-1}$, using also the bounds on $R(\xi)$ in Lemma \ref{lem:Wrons}.

 For $\xi>1$, by Lemma \ref{lem:Jost3} we compute from \eqref{eq:T_def} and \eqref{eq:Refl_def} directly that
  \begin{equation}
 \label{eq:RTxibig}
    T(\xi)=1+O(\xi^{-1}),\quad R(\xi)=O(\xi^{-1})\qquad \xi\to\infty.
  \end{equation} 
  With Lemma \ref{lem:Jost3} the claimed bounds follow.
\end{proof}

We now state a  useful corollary about the vanishing at $\xi=0$ of the distorted Fourier transform 
\begin{equation}\label{eq:dFT_def}
    \widetilde{\mathcal{F}}(f)(\xi)=\widetilde{f}(\xi):=\int f(x)\overline{e(x,\xi)}dx.
\end{equation}

\begin{cor}
    \label{cor:vanish}
    There exists $C>0$, such that for all $\abs{\xi}\leq 1$ and $f\in \jap{x}^{-2}L^1(\R)$ one has 
    \begin{equation}
    \abs{\widetilde f(\xi)}\le C\xi^2 \norm{\langle x\rangle^2 f}_{L^1},\qquad \abs{\p_\xi \widetilde{f}(\xi)}\leq C\xi \norm{\jap{x}^2f}_{L^1}.
    \end{equation}
\end{cor}
\begin{proof}
    This follows immediately from the bounds on $a,b$ in the previous proposition. 
\end{proof}

We also note the standard fact that Plancherel's identity holds for the distorted Fourier transform \eqref{eq:dFT_def}, and its inverse is given by its adjoint $\widetilde{\mathcal{F}}^*$, i.e., one has on $L^2(\R)$
\begin{equation}
    \jap{f,g}_{L^2}=\jap{\widetilde{\mathcal{F}}f,\widetilde{\mathcal{F}}g}_{L^2},\qquad f=\widetilde{\mathcal{F}}^*\widetilde{\mathcal{F}}f,\qquad \widetilde{\mathcal{F}}^*f(x)=\int f(\xi)e(x,\xi)d\xi.
\end{equation}

\subsection{Bounds on the linear evolution}

We begin with an $L^2$-bound on pseudo-differential operators. 

\begin{lem}
    \label{lem:CalVal}
Let $m=m(x,\xi)$ be smooth on $\R_x\times (\R_\xi\setminus\{0\})$, satisfying bounds 
\begin{equation}\label{eq:CalVal_bds}
|(\langle x\rangle \p_x)^k  m(x,\xi)|+ | (\xi\p_\xi)^\ell m(x,\xi)|\le B,\qquad k,\ell\in \{0,1,2  \}.
\end{equation}
Then 
\[
f\mapsto(T_m f)(\xi) :=\int e^{ix\xi}\; m(x,\xi) f(x)\, dx
\]
is bounded on $L^2(\R)$ with norm bound $\norm{T_{m}}_{L^2 \to L^2}\leq C B$ for some $C>0$.
\end{lem}
\begin{proof}
    For $|\xi|\ge1$ this follows from standard facts, such as the Calderon-Vaillancourt theorem, see~\cite[Proposition 9.4]{MuS}. For $0<\xi\le 1$ (as usual it suffices to consider $\xi>0$) we decompose 
    $(0,1] = \bigcup_{j\geq 0} I_j$, where $I_j=(2^{-j-1},2^{-j}]$ for $j\ge0$ . We choose an adapted partition of unity 
    \[
    1=\sum_{j\geq 0} \chi_j(\xi),\quad 0<\xi\le 1,
    \]
with non-negative $\chi_j\in C^\infty_c$ and $\textnormal{supp}(\chi_j)\subset 2I_j$, where $2I_j$ is the interval with same center $\xi_j$ as $I_j$ but twice its length, $\abs{2I_j}=2\abs{I_j}$. 
    Then we have the Fourier series expansion
    \[
    m_j(x,\xi) := m(x,\xi)\chi_j(\xi) = \sum_{n=0}^\infty e^{\pi i(\xi+\xi_j)n/|I_j|} c_{n,j}(x), 
    \]
    with $\|c_{n,j}\|_{L^\infty}\les B\langle n\rangle^{-2}$ for all $n\ge0$ by the bound \eqref{eq:CalVal_bds} on derivatives in~$\xi$.  Hence,
    \begin{align*}
   (T_{m_j} f)(\xi) &= \sum_{n=0}^\infty e^{\pi i(\xi+\xi_j)n/|I_j|}  \int e^{ix\xi} c_{n,j}(x)f(x)\, dx \\
   &= \sum_{n=0}^\infty e^{\pi i(\xi+\xi_j)n/|I_j|}  \widehat{c_{n,j}f}(\xi).
    \end{align*}
    We conclude that 
    \[
    \sup_{j\geq0} \| T_{m_j}\|_{L^2 \to L^2}\le C B.
     \]
    Moreover,  $T_{m_j}^* T_{m_k}=0$ if $\abs{k-j}>2$ (i.e., $I_j,I_k$ are not nearest or next-to-nearest neighbors). 
    On the other hand, for the same choice of $j,k$ we have
     \begin{align*}
         (T_{m_j} T_{m_k}^* f)(\xi)
         &= \int K_{j,k}(\xi,\eta) f(\eta)\, d\eta,\qquad 
         K_{j,k}(\xi,\eta)=\chi_j(\xi)\chi_k(\eta)\int e^{ix(\xi-\eta)}\; m(x,\xi) \overline{m(x,\eta)}\, dx,
     \end{align*}
    where the integral exists in the sense of an indefinite integral, i.e., as limit $L\to\infty$ of the integral over $|x|\le L$. The kernel $K_{j,k}$ satisfies
    \begin{align*}
     |K_{j,k}(\xi,\eta)| &\les \chi_j(\xi)\chi_k(\eta)B^2\min(L, \abs{\xi_j-\xi_k}^{-2}L^{-1} ) \\
     & \les \chi_j(\xi)\chi_k(\eta)\min\{2^{j},2^k\}B^2.   
    \end{align*}
    In fact, the bound by $L$ is derived by passing absolute values inside for $|x|\le L$, while $L^{-1}$ is obtained after two integrations by parts on $|x|\gtrsim L$. The final bound follows by optimizing in~$L$. 
    By Schur's test therefore
    \begin{equation*}
    \| T_{m_j} T_{m_k}^*\|_{L^2 \to L^2}\les 2^{-\frac{j}{2}-\frac{k}{2}}\min\{2^j,2^k\}B^2.
    \end{equation*}
    The Cotlar-Stein lemma (see e.g.~\cite[Chapter 9]{MuS}) finishes the proof. 
\end{proof}

Returning to the distorted Fourier transform relative to our Schr\"odinger operator~$\calL$, we first prove the following natural bounds. 

\begin{lem}\label{lem:derivbounds} For any $f\in C^1(\R)$ with compact support, 
\begin{align*}
    \|\tilde f'\|_{L^2} &\les \| \langle x\rangle f\|_{L^2},\quad \| f'\|_{L^2} \les \| \xi \tilde f\|_{L^2},\\
    \|xf\|_{L^2} &\les \|\tilde f'\|_{L^2},\quad \|\xi\tilde f\|_{L^2} \les \norm{\p_xf}_{L^2}+\|\jap{x}^{-1}f\|_{L^2}.
\end{align*}
\end{lem}
\begin{proof}
Using that $e(x,\xi)=e(-x,-\xi)$ (see also \eqref{eq:ab-symm-def}) we may assume that $\tilde f(\xi)=0$ for $\xi\le0$.

We then have
    \begin{align*}
       \tilde f'(\xi) &= \int \p_\xi \ol{e(x,\xi)} f(x)\, dx \\
       & =  \int \Big[ -ix\langle x\rangle^{-1} e^{-ix\xi} \ol{a(x,\xi)} +ix\langle x\rangle^{-1}  e^{ix\xi} \ol{b(x,\xi)} \Big] \langle x\rangle f(x)\, dx \\
       & \qquad + \int \Big[ \langle x\rangle^{-1} e^{-ix\xi} \p_\xi \ol{a(x,\xi)} + \langle x\rangle^{-1}  e^{ix\xi} \p_\xi\ol{b(x,\xi)} \Big] \langle x\rangle f(x)\, dx,
    \end{align*}
    where by Proposition~\ref{prop:dFT} each of the symbols $m=x\jap{x}^{-1}\overline{d(x,\xi)}$ or $m=\jap{x}^{-1}\p_\xi\overline{d(x,\xi)}$, where $d\in\{a,b\}$, satisfies the assumptions \eqref{eq:CalVal_bds} of Lemma~\ref{lem:CalVal}, and the first bound thus follows. For the second bound we have
    \begin{align*}
        f'(x) = \int_0^\infty \p_x e(x,\xi) \tilde f(\xi)\, d\xi 
       & =  \int_0^\infty  \Big[ i  e^{ix\xi} {a(x,\xi)} -i  e^{-ix\xi} {b(x,\xi)} \Big]   \xi \tilde f(\xi)\, d\xi \\
       & \qquad + \int_0^\infty \Big[ e^{ix\xi} \xi^{-1}\p_x {a(x,\xi)} +  e^{-ix\xi} \xi^{-1} \p_x{b(x,\xi)} \Big] \xi \tilde f(\xi)\, d\xi.
    \end{align*}
    By Proposition~\ref{prop:dFT}, Lemma~\ref{lem:CalVal} also applies with the symbols $\xi^{-1}\p_xa,\xi^{-1}\p_xb$.   
    For the final bound, we write 
    \begin{align*}
        xf(x)&= \int_0^\infty  \Big[ i  e^{ix\xi} \xi\p_\xi {a(x,\xi)} -i  e^{-ix\xi} \xi\p_\xi {b(x,\xi)} \Big] \xi^{-1}\tilde f(\xi)    \, d\xi \\
       & \qquad + i\int_0^\infty \Big[ e^{ix\xi}  {a(x,\xi)} -  e^{-ix\xi} {b(x,\xi)} \Big] \p_\xi \tilde f(\xi)\, d\xi
    \end{align*}
    whence again by Proposition~\ref{prop:dFT} and Lemma~\ref{lem:CalVal}
    \begin{equation*}
    \|xf\|_{L^2}\les \|\xi^{-1}\tilde f\|_{L^2} + \|\tilde f'\|_{L^2}\les \|\tilde f'\|_{L^2}
    \end{equation*}
    as claimed. Finally, the bound for $\xi\tilde f(\xi)$ follows similarly upon using that
    \begin{equation}
     \xi \overline{e(x,\xi)}=\xi \overline{a(x,\xi)}e^{-ix\xi}+\xi\overline{b(x,\xi)}e^{ix\xi}= i\overline{a(x,\xi)}\p_x\left(e^{-ix\xi}\right)-i\overline{b(x,\xi)}\p_x\left(e^{ix\xi}\right),
    \end{equation}
    integrating by parts and invoking Proposition~\ref{prop:dFT} and Lemma~\ref{lem:CalVal}.
\end{proof}

Next, we obtain the following related statement, cf.\ Proposition~3.1 in Germain-Pusateri-Rousset~\cite{GPR}. Note that our potential fails their condition $V\in L^1_1$. 

\begin{lem}
    \label{thm:dispFT}
    For any $f \in \jap{x}^{-2}L^1(\R)$ one has 
    \[
    \| e^{it\calL} f\|_{L^\infty} \les t^{-\frac12} \|\tilde f\|_{L^\infty} + t^{-\frac34} \|\tilde f'\|_{L^2}
    \]
    for all $t>0$.
\end{lem}
\begin{proof}
Let us first assume that $\widetilde f(\xi)=0$ if $\xi\le0$. 
    Then Proposition~\ref{prop:dFT} implies that
    \begin{equation}\label{eq:eitL0}
    \begin{aligned}
        \big(e^{it\calL} f\big)(x) &= \int e^{it\xi^2}  {e(x,\xi)}\tilde f(\xi)\, d\xi \\
        &= \int e^{it\xi^2} \big[  e^{ix\xi}  {a(x,\xi)} + e^{-ix\xi}  {b(x,\xi)} \big]  \tilde f(\xi)\, d\xi \\
         &= e^{-ix^2/(4t)} \int e^{it(\xi+x/(2t))^2}   {a(x,\xi)}  \tilde f(\xi)\, d\xi +  e^{-ix^2/(4t)} \int e^{it(\xi-x/(2t))^2}  {b(x,\xi)}   \tilde f(\xi)\, d\xi.
    \end{aligned}
    \end{equation}
By \cite[Lemma 2]{2Dim}, with $\delta=t^{-\frac12}$ and $x_0=\frac{x}{2t}$, 
\begin{equation}\label{eq:eitL}
     \begin{aligned}
       \Big| \big(e^{it\calL} f\big)(x)\Big| 
         &\le \delta^2 \Big(  \int   \frac{|{a(x,\xi) } \tilde f(\xi)|}{\delta^2+(\xi-x_0)^2}\, d \xi +\int_{|\xi-x_0|>\delta}   \frac{|\p_\xi({a(x,\xi) }  \tilde f(\xi))|}{|\xi-x_0|}\, d \xi \Big)   \\
         & \qquad+ \delta^2 \Big(  \int   \frac{|{b(x,\xi) } \tilde f(\xi)|}{\delta^2+(\xi+x_0)^2}\, d \xi +\int_{|\xi+x_0|>\delta}   \frac{|\p_\xi({b(x,\xi)}  \tilde f(\xi))|}{|\xi+x_0|}\, d \xi \Big)\\
         &\les t^{-\frac12} \|\tilde f\|_{L^\infty} + t^{-\frac34}\| \tilde f'\|_{L^2}.
    \end{aligned}
\end{equation}
To pass to the final estimate, we have used the bounds on $a,b$ from Proposition~\ref{prop:dFT}, applied Cauchy-Schwarz and the bound (for both $a$ and $b$)
\begin{equation*}
\| \p_\xi({a(x,\xi)}  \tilde f(\xi))\|_{L^2}+\| \p_\xi({b(x,\xi)}  \tilde f(\xi))\|_{L^2} \les \|\xi^{-1}\tilde f\|_{L^2}+ \|\tilde f'\|_{L^2} \les \|\tilde f'\|_{L^2},
\end{equation*}
using Hardy's inequality (see e.g.~\cite[Lemma 2]{Gavin}) for the last inequality. For general functions $f\in \jap{x}^{-2}L^1(\R)$ we write
$
\tilde f = \one_{\R+} \tilde f + \one_{\R^-} \tilde f,
$
and use Corollary~\ref{cor:vanish} to see that the sharp cutoff does not affect the proof. 
\end{proof}

Next, we turn to Lemma~12 of~\cite{Gavin}.

\begin{lem}
    \label{lem:Lem12} 
   There exists a constant $C>0$ so that for all $|x|\le t^\gamma$, with $\frac12\le\gamma\le 1$, 
    \begin{equation}
        \label{eq:localdec}
        |(e^{it\calL} f)(x)|\le Ct^{-\frac32(1-\gamma)} \| \tilde f'\|_{L^2}
    \end{equation}
    for all $f\in \jap{x}^{-2}L^1(\R)$. 
\end{lem}

\begin{proof}
As in the proof of the previous lemma, it will suffice to consider $\tilde f$ supported on $\xi>0$. 
    By~\eqref{eq:eitL0}, with $x_0=\frac{x}{2t}\in [-M,M]$,
    \begin{align*}
        |(e^{it\calL} f)(x)| &\le \Big| \int e^{it\lambda^2}  {a(x,\lambda-x_0)}  \tilde f(\lambda-x_0)\chi(\lambda/L)\, d\lambda\Big|\\
         &\qquad +  \Big| \int e^{it\eta^2}  {b(x,\eta+x_0)}   \tilde f(\eta+x_0)\chi(\eta/L)\, d\eta \Big| \\
         &\qquad + \Big| \int e^{it\lambda^2}  {a(x,\lambda-x_0)}  \tilde f(\lambda-x_0)(1-\chi(\lambda/L))\, d\lambda\Big|\\
         &\qquad +  \Big| \int e^{it\eta^2}  {b(x,\eta+x_0)}   \tilde f(\eta+x_0)(1-\chi(\eta/L))\, d\eta \Big|,
    \end{align*}
where as usual $\chi\in C^\infty_c$ is a smooth cutoff function localizing to the interval $[-1,1]$. The first two lines are bounded by 
\begin{align*}
     \int_{|\xi|\les M+L} |\tilde f(\xi)|\, d\xi &\les \int_0^1 \int_{|\xi|\les M+L}  |\tilde f'(s\xi)||\xi|\, ds d\xi\\
     &\les (M+L)^{\frac32} \int_0^1 \|\tilde f'\|_{L^2} \, s^{-\frac12}\, ds \les (M+L)^{\frac32} \|\tilde f'\|_{L^2}.
\end{align*}
    On the other hand, integrating by parts in the third and fourth lines shows that they are bounded by 
    \begin{align*}
      &  t^{-1}\Big| \int_{|\lambda|>L}  |\lambda|^{-1} \big| \partial_\lambda  [ a(x,\lambda-x_0)  \tilde f(\lambda-x_0)(1-\chi(\lambda/L)) ]\big| \, d\lambda \\
         &\qquad +  t^{-1}  \int_{|\eta|>L} |\eta|^{-1} \big| \p_\eta [ b(x,\eta+x_0)   \tilde f(\eta+x_0)(1-\chi(\eta/L))]  \big|\, d\eta \\
         &\les t^{-1}L^{-\frac12} (\|\xi^{-1}\tilde f\|_{L^2} + \|\tilde f'\|_{L^2})\les  t^{-1}L^{-\frac12}\|\tilde f'\|_{L^2},
    \end{align*}
    provided that $L\ge 2M$. Setting $L=2M=t^{\gamma-1}$ we obtain the stated bound.
\end{proof}

In practice, we will often use this lemma in the form
\begin{equation}
 \norm{f}_{L^\infty(\abs{x}\leq t^\gamma)}\leq C t^{-\frac32(1-\gamma)}\|\partial_\xi\widetilde{\mathcal{F}}(e^{-it\calL} f)\|_{L^2}=C t^{-\frac32(1-\gamma)}\norm{J_V(t)f}_{L^2},
\end{equation}
where $J_V(t)$ is the distorted Galilean vector field, see \eqref{eq:Galilean_VFs}.

We will also need the following decay estimate.

\begin{lem}
    \label{lem:microdec}
    For any $\mu>0$ there exists $C_\mu>0$ such that for $f\in \jap{x}^{-2}L^1(\R)$ there holds
    \begin{equation}
        \label{eq:micro}
        \Big\| \langle x\rangle^{-\frac12-\mu}\int e^{it\xi^2} e(x,\xi) \langle x\xi\rangle^{-\frac12} \tilde f(\xi)\, d\xi\Big\|_{L^\infty} \le C_\mu\,  t^{-1}\big( \|f\|_{L^2} + \|\tilde f'\|_{L^2}\big)
    \end{equation}
    for all $t>0$. 
\end{lem}
\begin{proof}
    By Corollary~\ref{cor:vanish} it suffices to consider $\xi>0$, and by a smooth partition of unity we distinguish $\xi>1$ from $0<\xi\les 1$. We begin with the latter, simply assuming that $\tilde f(\xi)=0$ if $\xi\ge1$. Then 
    \begin{align*}
        &\langle x\rangle^{-\frac12-\mu} \int_0^\infty e^{it\xi^2} e(x,\xi) \langle x\xi\rangle^{-\frac12} \tilde f(\xi)\, d\xi = -\langle x\rangle^{-\frac12-\mu} (2it)^{-1} \int_0^\infty e^{it\xi^2} \p_\xi\big[ \xi^{-1}e(x,\xi) \langle x\xi\rangle^{-\frac12} \tilde f(\xi)\big]\, d\xi.
    \end{align*}
    By Proposition~\ref{prop:dFT} we have
    \[
    \langle x\rangle^{-\frac12-\mu} \xi^{-1} |e(x,\xi)| \langle x\xi\rangle^{-\frac12} \les \xi^{-\frac12+\mu}
    \]
    and 
     \[
    \langle x\rangle^{-\frac12-\mu} \big| \p_\xi \big[ \xi^{-1} e(x,\xi) \langle x\xi\rangle^{-\frac12} \big]\big|  \les \xi^{-\frac32+\mu}.
    \]
    These bounds imply \eqref{eq:micro} by Cauchy-Schwarz.  The bounds in case of $\xi\gtrsim 1$ are more favorable and the corresponding details left to the reader. 
\end{proof}

\subsection{Finer asymptotics for the distorted Fourier basis}
The nonlinear analysis requires a more accurate expansion of $R(\xi)$ for small~$\xi$, cf.~Lemma~\ref{lem:Wrons}. For this 
we will need an asymptotic expansion of the function~$\rho_1(u,\xi)$ in Lemma~\ref{lem:Jost1}.

\begin{lem}
\label{lem:pain}
    With the notation of Lemma~\ref{lem:Jost1}, there exists $C>0$ such that for $u\geq 2C\xi$ one has an absolutely  convergent expansion
    \begin{align}\label{eq:rho1exp}
        \rho_1(u,\xi) &= \sum_{j=0}^\infty \alpha_j(u) \xi^{2j},\qquad \alpha_j(u)=(-1)^ju^{-2-2j}\gamma_j(u),
    \end{align}
    where $\gamma_j(u)$ are smooth functions satisfying
    \begin{equation}\label{eq:gammaj_bd}
    |(u\p_u)^a\gamma_j(u)|\les_a C^{2j}\jap{j}^a,\qquad a\in\mathbb{N}_0.
    \end{equation}
\end{lem}

\begin{proof}
    Returning to the proof of Lemma~\ref{lem:Jost1}, we write the Volterra integral equation~\eqref{eq:rho1} in the form
    \begin{align}\label{eq:Volt_rho1exp}
        \rho_1(u,\xi) &= \int_0^\infty \Big[\big(1+e^{2is}\big)s -i \big(1-e^{2is}\big)(1+u(u+s))\Big] \frac{u+s+i}{u+i} \;\frac{1+\xi^2\rho_1(u+s,\xi)}{(u+s)^6\Big(1+\frac{\xi^2}{(u+s)^2}\Big)} \, ds.
    \end{align}
Expanding for $s\geq 0$ and $u\geq 2\xi$
\begin{equation*}
\frac{1}{1+\frac{\xi^2}{(u+s)^2}} = \sum_{j=0}^\infty (-1)^j \frac{\xi^{2j}}{(u+s)^{2j}}
\end{equation*}
yields that
\begin{equation}
  \rho_1(u,\xi)=\sum_{j\geq 0}(-1)^j\int_0^\infty K_0(u,u+s)\frac{u+s+i}{(u+i)(u+s)^{6+2j}} \left(1+\xi^2\rho_1(u+s,\xi)\right)ds.
\end{equation}
Plugging in the ansatz \eqref{eq:rho1exp}, we obtain with
\begin{equation}
    G_j(u,s):=K_0(u,u+s)\frac{u+s+i}{(u+i)(u+s)^{6+2j}}
\end{equation}
that
\begin{equation}
    \alpha_{\ell}(u)=(-1)^\ell\int_0^\infty G_\ell(u,s)ds+\sum_{k=0}^{\ell-1}(-1)^{\ell-k-1}\int_0^\infty G_{\ell-k-1}(u,s)\alpha_k(u+s)ds.
\end{equation}
Hence $\gamma_\ell(u):=(-1)^\ell u^{2+2\ell}\alpha_\ell(u)$ satisfies
\begin{equation}\label{eq:gamma_rec}
    \gamma_{\ell}(u)=\int_0^\infty \widetilde{G}_\ell(u,s)ds-\sum_{k=0}^{\ell-1}\int_0^\infty \widetilde{G}_{\ell}(u,s)\gamma_k(u+s)ds,
\end{equation}
where
\begin{equation}
  \widetilde{G}_{\ell}(u,s)=K_0(u,u+s)\frac{u+s+i}{(u+i)(u+s)^{4}}\left(\frac{u}{u+s}\right)^{2+2\ell}.  
\end{equation}
We note that whenever $g$ is a bounded function, then so is
\begin{equation}
  F_\ell[g](u):=\int_0^\infty \widetilde{G}_\ell(u,s)g(u+s)ds.
\end{equation}
This follows from the explicit formula of $\widetilde{G}_\ell$, using for $0<u,s\lesssim1$ that by \eqref{eq:cubic} we have $K_0(u,u+s)=O(s^3)$ and the general bound $\abs{K_0(u,u+s)}\leq 1+s+u(u+s)$. Similarly, from the explicit formula and \eqref{eq:dudvK0} it is clear that if $g$ has symbol behavior, then so does $F_\ell[g](u)$ with  $|(u\p_u)^aF_\ell[g](u)|\les_a \jap{\ell}^a$ for $a\in\mathbb{N}_0$.

Finally, letting
\begin{equation}
    \gamma_{\ell0}(u):=\int_0^\infty \widetilde{G}_\ell(u,s)ds,
\end{equation}
we note that $\gamma_0(u)=\gamma_{00}(u)$, and the recursive formula \eqref{eq:gamma_rec} together with aforementioned mapping properties of $F_\ell$ proves the claim for $\gamma_\ell(u)$.
\end{proof}

We are now able to state the following result. 

\begin{cor}\label{cor:pain}
In place of \eqref{eq:c2weak} one has  for small $\xi>0$
\begin{equation} 
 \label{eq:c2strong}
 c_2(\xi) = -3i\xi^{-1}(1+O(\xi))
 \end{equation}
   as well as 
    \begin{equation}\label{eq:Rstrong}
           R(\xi)= 1 + O(\xi) ,
    \end{equation}
    where the $O(\xi)$-terms exhibit symbol-type behavior.  The expansions \eqref{eq:tilc1} and~\eqref{eq:tilc2} hold in the same stronger form.  
\end{cor}
\begin{proof}
We deduce from Lemma~\ref{lem:pain} that for $u=x\xi$ with $x\in[C_*,\delta\xi^{-1}]$
    \begin{align}\label{eq:rho1_exp}
         \xi^2\rho_1(u,\xi) &=x^{-2}F(x,u), 
    \end{align}
where $F(u,x)$ is smooth and bounded, and of symbol-type in both variables for $x\gg1$. For $c_2(\xi)=-W[f_+,\varphi_1]$ as defined in~\eqref{eq:c2}, we have
\begin{equation}\label{eq:c2_expanded}
\begin{aligned}
  e^{iu}\xi c_2(\xi)&= \left(1+\frac{i}{u}\right) (1+\xi^2\rho_1(u,\xi)) \xi\partial_x\varphi_1\\&-\left(i\xi^2 \big(1+\frac{i}{u}\big) (1+\xi^2\rho_1(u,\xi))- \frac{i}{x^2} (1+\xi^2\rho_1(u,\xi))+\big(1+\frac{i}{u}\big) \xi^4 \p_u\rho_1(u,\xi)\right)\varphi_1.
\end{aligned}  
\end{equation}
Since $\xi\partial_x\varphi_1=2u+O(u^3)$ and $\xi^2\varphi_1=O(u^2)$ it follows with \eqref{eq:rho1_exp} that
\begin{equation}
 \xi c_2(\xi)= 3i+x^{-2}F(x,u)+O(u),
\end{equation}
where $F(x,u)$ is of symbol-type and uniformly bounded for $C_*\leq x\leq \delta\xi^{-1}$. Choosing $x=M\gg 1$ and thus $u=M\xi$, we argue by contradiction as $\xi\searrow 0$ to obtain \eqref{eq:c2strong}, and the symbol properties follow from \eqref{eq:c2_expanded}. From this, \eqref{eq:Rstrong} follows by recalling from \eqref{eq:Refl} that
\begin{equation}
    R(\xi)=-\frac{\overline{c_2(\xi)}}{c_2(\xi)} (1+O(\xi^2)),
\end{equation}
which implies the claimed property. 
\end{proof}

The finer bounds of this section lead to better bounds on the derivatives of $a,b$, cf.~Proposition~\ref{prop:dFT}. 
\begin{cor}
    \label{cor:finer ab}
    Using the notation of Proposition~\ref{prop:dFT}, for $|\xi|\le 1$ one has
    \begin{align*}
        |\p_x a(x,\xi)|+|\p_x b(x,\xi)| &\les \langle x\rangle^{-1}\langle x\xi\rangle^{-1}, \\
        |\p_\xi a(x,\xi)|+|\p_\xi b(x,\xi)| &\les 1+ |\xi|^{-1}\langle x\xi\rangle^{-1},
    \end{align*}
    as well as
    \[
     |\p_x a(x,\xi)|+|\p_x b(x,\xi)| \les \xi^2\langle x
     \rangle \one_{[|x\xi|\le 1]} + x^{-2}|\xi|^{-1}\one_{[|x\xi|\ge1]}.
    \]
    For $|\xi|\ge1$, 
     \begin{align*}
        |\p_x a(x,\xi)|+|\p_x b(x,\xi)| &\les |\xi|^{-1}\langle x\rangle^{-2},\\
        |\p_\xi a(x,\xi)|+|\p_\xi b(x,\xi)| &\les \xi^{-2}.
    \end{align*}
    These bounds are stable under arbitrary action by $\xi\p_\xi$ and $\langle x\rangle\p_x$. 
\end{cor}
\begin{proof}
    The proof proceeds as in the case of Proposition~\ref{prop:dFT}. We assume without loss of generality that $\xi>0$. 
    We consider first $0<\xi\leq 1$. Then the claimed bounds are a direct consequence of Proposition~\ref{prop:dFT} in case $x\xi>-\delta$. For $x\xi<-\delta$, we invoke the sharper bounds on $\rho_1$ and $R$ provided by Lemma~\ref{lem:pain} and Corollary~\ref{cor:pain}. Using also that the first term in the expression of $a$ in \eqref{eq:def_ab_neg} vanishes if $x<-\xi^{-1}$, we obtain that $\abs{x\p_xa}\les \jap{x\xi}^{-2}$ and $\abs{\xi\p_\xi a}\les \jap{x\xi}^{-2}$, as well as $\abs{x\p_xb}\les \jap{x\xi}^{-2}$ and $\abs{\xi\p_\xi b}\les \xi+\jap{x\xi}^{-1}$. When $\xi>1$, the bounds follow via Lemma \ref{lem:Jost3}.
\end{proof}

\section{Comparison between $J_0(t)$ and $J_V(t)$, and between the distorted and standard frequencies}
\label{sec:J0JV}

We now turn to the main issue in Stewart's paper~\cite{Gavin}, namely the comparison between the distorted Galilei vector field $J_V(t)$ and the free one $J_0(t)$, where
\begin{equation}\label{eq:Galilean_VFs}
\begin{aligned}
   J_V(t) &= e^{it\calL}\tcalF^* i\p_\xi \tcalF  e^{-it\calL}, \\
 J_0(t) &= e^{-it\p_x^2}\calF^* i\p_\xi \calF  e^{it\p_x^2} = x - 2it\p_x.
\end{aligned}
\end{equation}
Here, $\tcalF$ and $\calF$ are the distorted, respectively, standard Fourier transforms. 
\begin{lem}\label{lem:directJVJO} There is a linear, $L^2$-bounded operator $T_0$ such that for any $f\in C^1(\R)$ of compact support 
   \begin{align*}
   \tcalF\left(\big[ J_V(t) - T_0 J_0(t)\big]f\right)(\xi) &= t \int_{-\infty}^{\infty} \big(K_1(x,\xi)e^{ix\xi} + K_2(x,\xi) e^{-ix\xi}\big)f(x)\,dx\\
   &\qquad  + \int_{-\infty}^{\infty} \big(K_3(x,\xi) e^{ix\xi} + K_4(x,\xi) e^{-ix\xi}\big) f(x)\,dx,
   \end{align*}
   where the kernels $K_j(x,\xi)$ -- given below in \eqref{eq:kernelsKj} -- satisfy the bounds for $|\xi|\le 1$
   \begin{align*}
  \sum_{j=1,2}\big|K_j(x,\xi)\big| &\lesssim \langle x\rangle^{-1}\cdot\langle x\xi\rangle^{-1}\one_{\abs{x\xi}\gtrsim 1}+\xi^2\jap{x}\one_{\abs{x\xi}\les 1},\\
  \sum_{j=3,4}\big|K_j(x,\xi)\big| &\lesssim 1 + |\xi|^{-1}\cdot\langle x\xi\rangle^{-1},
  \end{align*}
  while for $|\xi|\ge1$
     \begin{align*}
  \sum_{j=1,2}\big|K_j(x,\xi)\big| &\lesssim |\xi|^{-1}\langle x\rangle^{-2},\quad
  \sum_{j=3,4}\big|K_j(x,\xi)\big| \lesssim   |\xi|^{-2}.
  \end{align*}
   Moreover, these bounds are stable under arbitrary applications of  $\langle x\rangle\partial_x$, $\xi\partial_{\xi}$ to $K_j$. 
  \end{lem}
\begin{proof} By definition,  
\[
\tcalF[J_V(t){f}](\xi) = \big(i\partial_{\xi} + 2t\xi\big)\tilde{f}(\xi) = \int_{-\infty}^\infty \big(i\partial_{\xi} +2t\xi\big)\overline{e(x,\xi)} f(x)\,dx. 
\]
For $\xi>0$, in light of Proposition~\ref{prop:dFT},  
\begin{align*}
\big(i\partial_{\xi} + 2t\xi\big)\overline{e(x,\xi)} &= \overline{a(x,\xi)}\big(x + 2it\partial_x\big)e^{-ix\cdot \xi} -  \overline{b(x,\xi)}\big(x + 2it\partial_x\big)e^{ix\cdot \xi}\\
&\qquad + i\partial_{\xi}\overline{a(x,\xi)}\, e^{-ix\cdot \xi}+ i\partial_{\xi}\overline{b(x,\xi)}\, e^{ix\cdot \xi}\\
& =A(x,\xi) + B(x,\xi)
\end{align*}
where $B$ is the expression on the second line. 
Integrating by parts yields 
\begin{align*}
\int_{-\infty}^\infty A(x,\xi) f(x)\,dx &= (\tcalF T_0 J_0(t) f)(\xi)-2it\int_{-\infty}^{\infty}\big(\partial_x\overline{a(x,\xi)} - \partial_x\overline{b(x,\xi)}\big) f(x)\,dx
\end{align*}
where 
\[
T_0 g(\xi): = \tcalF^*\int_{-\infty}^{\infty}\big(\overline{a(x,\xi)}\, e^{-ix\cdot \xi} - \overline{b(x,\xi)}\, e^{ix\cdot \xi}\big) g(x)\,dx,
\]
is an $L^2$-bounded operator in view of Lemma~\ref{lem:CalVal} as well as Proposition~\ref{prop:dFT}. By inspection
\begin{equation}\label{eq:kernelsKj}
\begin{aligned}
    K_1(x,\xi) &= 2i \partial_x\overline{b(x,\xi)},\quad  K_2(x,\xi) =  -2i \partial_x\overline{a(x,\xi)}\\
     K_3(x,\xi) &=  i\partial_{\xi}\overline{b(x,\xi)} ,\quad  K_4(x,\xi) = i\partial_{\xi}\overline{a(x,\xi)}.
\end{aligned}
\end{equation}
The stated kernel bounds are then a consequence of Corollary \ref{cor:finer ab}. 
 \end{proof}   

Next, consider the reverse order. 

\begin{lem}\label{lem:directJ_0J_V} There exists a linear, $L^2$-bounded operator $T_V$ such that we have the following kernel representation: for $f\in C^1(\R)$ with compact support
\begin{align*}
\big(\big[J_0(t) - T_V J_V(t)\big]f\big)(x) &= t \int_{-\infty}^{\infty} \big(S_1(x,\xi)e^{ix\xi} + S_2(x,\xi) e^{-ix\xi}\big)\tilde{f}(\xi)\,d\xi\\
   &\qquad  + \int_{-\infty}^{\infty} \big(S_3(x,\xi) e^{ix\xi} + S_4(x,\xi) e^{-ix\xi}\big) \tilde{f}(\xi)\,d\xi,
\end{align*}
where $S_1=-\ol{K_2}$, $S_2=\ol{K_1}$, $S_3=-\ol{K_4}$, and $S_4=\ol{K_3}$ with $K_j$ the kernels of Lemma~\ref{lem:directJVJO}.
\end{lem}
\begin{proof} This is similar to the preceding proof. One starts by writing 
\begin{align*}
\big(x - 2it\partial_x\big)f&=\big(x - 2it\partial_x\big)\int_{-\infty}^{\infty}\big(a(x,\xi)e^{ix\xi} + b(x,\xi)e^{-ix\xi}\big) \tilde{f}(\xi)\,d\xi\\
&=\int_{-\infty}^{\infty}\big(a(x,\xi)(-i)(\partial_{\xi} + 2it\xi)e^{ix\xi} + b(x,\xi)i(\partial_{\xi} + 2it\xi)e^{-ix\xi}\big) \tilde{f}(\xi)\,d\xi\\
&\qquad + t \int_{-\infty}^{\infty}\big(-2i\partial_xa(x,\xi) e^{ix\xi} - 2i\partial_xb(x,\xi) e^{-ix\xi}\big) \tilde{f}(\xi)\,d\xi.
\end{align*}
Denoting the last two integral expressions as $
\big[\big(x - 2it\partial_x\big)f]_{\ell}$, $\ell = 1, 2$, 
respectively, one integration by parts yields
\begin{align*}
&\big[\big(x - 2it\partial_x\big)f]_1 = T_V\big(i\tcalF^*(\partial_{\xi} - 2it\xi)\tilde{f}\big) +\int_{-\infty}^{\infty}\big(i\partial_{\xi}a(x,\xi) e^{ix\xi} - i\partial_{\xi}b(x,\xi) e^{-ix\xi}\big)  \tilde{f}(\xi)\,d\xi,
\end{align*}
where we define the operator $T_V$ by means of 
\[
T_V g(x) = \int_{-\infty}^{\infty}\big(a(x,\xi) e^{ix\xi} - b(x,\xi) e^{-ix\xi}\big) \tilde{g}(\xi)\,d\xi.
\]
This is an $L^2$-bounded operator due to the Plancherel theorem for the distorted Fourier transform,   Lemma~\ref{lem:CalVal}, and Proposition~\ref{prop:dFT}. Collecting terms, one finds that 
\begin{align*}
&\int_{-\infty}^{\infty}\big(i\partial_{\xi}a(x,\xi) e^{ix\xi} - i\partial_{\xi}b(x,\xi) e^{-ix\xi} \big) \tilde{f}(\xi)\,d\xi =:\int_{-\infty}^{\infty} \big(S_3(x,\xi) e^{ix\xi} + S_4(x,\xi) e^{-ix\xi}\big) \tilde{f}(\xi)\,d\xi,
\end{align*}
as well as
\begin{align*}
&\int_{-\infty}^{\infty}\big(-2i\partial_xa(x,\xi) e^{ix\xi} + 2i\partial_xb(x,\xi) e^{-ix\xi}\big) \tilde{f}(\xi)\,d\xi=:\int_{-\infty}^{\infty} \big(S_1(x,\xi) e^{ix\xi} + S_2(x,\xi) e^{-ix\xi}\big) \tilde{f}(\xi)\,d\xi.
\end{align*}
Comparison with the  proof of Lemma~\ref{lem:directJVJO} concludes the proof. 
\end{proof}

\begin{cor}\label{cor:J_0J_Vbound} We have the estimate 
\[
\big\|J_0(t) f\big\|_{L^2}\lesssim \big\|J_V(t)f\big\|_{L^2} + \big\|f\big\|_{L^2}.
\]
\end{cor}
\begin{proof}
To begin with, we observe that due to Hardy's inequality (see also Corollary \ref{cor:vanish}),
\begin{equation}
\label{eq:J0gen}
\big\|\xi^{-1} \tilde{f}(\xi)\big\|_{L^2} = \big\|\xi^{-1} e^{-it\xi^2}\tilde{f}(\xi)\big\|_{L^2}\lesssim \big\|\partial_{\xi}\big(e^{-it\xi^2}\tilde{f}\big)\big\|_{L^2}=\big\|J_V(t)f\big\|_{L^2}.
\end{equation}
 In view of the preceding lemma, and using Lemma~\ref{lem:CalVal}, we then deduce that 
\begin{align*}
\big\|\int_{-\infty}^{\infty} \big(S_3(x,\xi) e^{ix\xi} + S_4(x,\xi) e^{-ix\xi}\big) \tilde{f}(\xi)\,d\xi\big\|_{L^2_x}&\lesssim \big\|\xi^{-1}\tilde{f}(\xi)\big\|_{L^2} + \big\|\tilde{f}\big\|_{L^2}\\
&\lesssim \big\|J_V(t)f\big\|_{L^2} + \big\|f\big\|_{L^2}.
\end{align*}
For the contribution of the error term involving $S_{\ell}(x,\xi)$, $\ell=1,2$, we take advantage of the following identity:
\begin{align*}
\tilde{f}(\xi) = \frac{1}{2t\xi} \big(-i\partial_{\xi}\tilde{f}(\xi) + \tcalF(J_V(t)f)(\xi)\big).
\end{align*}
We conclude that 
\begin{align*}
&t\int_{-\infty}^{\infty} \big(S_1(x,\xi) e^{ix\xi} + S_2(x,\xi) e^{-ix\xi}\big) \tilde{f}(\xi)\,d\xi\\
&\quad = \int_{-\infty}^{\infty} (2\xi)^{-1}\big(S_1(x,\xi) e^{ix\xi} + S_2(x,\xi) e^{-ix\xi}\big) \tcalF(J_V(t)f)(\xi)\,d\xi\\
&\quad\qquad -  \int_{-\infty}^{\infty} (2\xi)^{-1}\big(S_1(x,\xi) e^{ix\xi} + S_2(x,\xi) e^{-ix\xi}\big)i\partial_{\xi}\tilde{f}(\xi)\,d\xi .
\end{align*}
The kernel bounds in the preceding Lemma \ref{lem:directJVJO} imply that 
\begin{align*}
\sum_{j=1,2}|\xi|^{-1}\big|S_j(x,\xi)\big|\lesssim 1,
\end{align*}
and similar bounds after applying powers of $\xi\partial_{\xi}$, $\langle x\rangle \partial_x$. It then follows from Lemma~\ref{lem:CalVal} and the Plancherel theorem for the distorted Fourier transform that 
\begin{align*}
&\big\|\int_{-\infty}^{\infty} (2\xi)^{-1}\big(S_1(x,\xi) e^{ix\xi} + S_2(x,\xi) e^{-ix\xi}\big) \tcalF(J_V(t)f)(\xi)\,d\xi\big\|_{L^2_{x}}\lesssim \big\|J_V(t)f\big\|_{L^2}. 
\end{align*}
For the remaining integral term involving $\partial_{\xi}\tilde{f}$, we perform one integration by parts, relying on the fact that 
\begin{align*}
\Big|\partial_{\xi}\Big((2\xi)^{-1}\big(S_1(x,\xi) e^{ix\xi} + S_2(x,\xi) e^{-ix\xi}\big)\Big)\Big|\lesssim \xi^{-1}. 
\end{align*}
This estimate remains correct after applying arbitrary powers of $\xi\partial_{\xi}$, $\langle x\rangle \partial_x$.
Lemma~\ref{lem:CalVal} then implies that 
\begin{align*}
&\big\|\int_{-\infty}^{\infty} (2\xi)^{-1}\big(S_1(x,\xi) e^{ix\xi} + S_2(x,\xi) e^{-ix\xi}\big)i\partial_{\xi}\tilde{f}(\xi)\,d\xi\big\|_{L^2_{x}}\lesssim \big\|\xi^{-1}\, \tilde{f}(\xi)\big\|_{L^2}\lesssim \big\|J_V(t)f\big\|_{L^2},
\end{align*}
where at the last step we again invoked~\eqref{eq:J0gen}. 
Combining the preceding observations with Lemma~\ref{lem:directJ_0J_V}, we deduce that 
\begin{align*}
\big\|J_0(t) f\big\|_{L^2}\lesssim \big\|T_V J_V (t)f\big\|_{L^2} + \big\|J_V(t)f\big\|_{L^2} + \big\|f\big\|_{L^2}\lesssim \big\|J_V(t)f\big\|_{L^2} + \big\|f\big\|_{L^2},
\end{align*}
as claimed. 
\end{proof}

In order to estimate the nonlinearity, we shall heavily rely on microlocal techniques, i.e., localizing both space and (distorted) frequency variables. An important role will be played by the fact that the {\it{distorted}} and the {\it{standard}} frequencies of a function, a priori unrelated, correlate closely provided the function is localized in space to the dual scale of the (small) distorted frequency scale. This is stated precisely in the following lemma, where we denote Littlewood-Paley-type frequency localization operators with respect to the distorted frequency by $P^V$,  while $P^0$ refers to the standard frequency.  To be precise, let $\psi_\lambda(\xi)=\psi(\xi/\lambda)\ge0$ where $\supp(\psi)\subset [-2,-1/2]\cup[1/2,2]$. We will assume that we have the Littlewood-Paley partition of unity, for any $\lambda>0$
\[
\sum_{j} \psi_{2^j\lambda}(\xi)=1,\qquad \forall \xi\ne0.
\]
 We will also use $$\psi_{<\lambda}=\sum_{j<0} \psi_{2^j\lambda}.$$
The operator $P_\lambda^V$ is then given by the distorted Fourier multiplier $\psi_\lambda$, and $P_{<\lambda}^V$ by the distorted  multiplier $\psi_{<\lambda}$. $P_\lambda^0$ and $P_{<\lambda}^0$ are defined analogously. 
 Moreover, $\phi_L(x)=\phi(x/L)$ where $\supp(\phi)\subset [-1,1]$. Both $\psi$ and~$\phi$ are smooth.

\begin{lem}
    \label{lem:freqcomp}
    For any $\lambda>0$ and $L\lambda\ge 10$, one has 
    \begin{equation}\label{eq:lamerr}
         \| \langle x\rangle P_{\lambda}^{V} \phi_L P_{<\lambda/4}^{0} \|_{L^2 \to L^2} +  \| \langle x\rangle P_{\lambda}^{V} \phi_L P_{>4\lambda}^{0} \|_{L^2\to L^2}\le C_N\, (1+L) (L\lambda)^{-N}
    \end{equation}
   for any $N\ge1$. Furthermore, 
    \begin{equation}\label{eq:lamerr2}
         \| \p_x P_{\lambda}^{V} \phi_L P_{<\lambda/4}^{0} \|_{L^2 \to L^2} +  \| \p_x P_{\lambda}^{V} \phi_L P_{>4\lambda}^{0} \|_{L^2\to L^2}\le C_N\, \lambda (L\lambda)^{-N}.
    \end{equation}
\end{lem}
\begin{proof}
    By Lemma~\ref{lem:derivbounds} one has
    \begin{equation}
    \label{eq:xP}
    \begin{aligned}
         \| x P_{\lambda}^{V} \phi_L P_{<\lambda/4}^{0}f\|_{L^2} &\les \| \p_\xi [\psi_\lambda(\xi) \tcalF(\phi_L P_{<\lambda/4}^{0}f) ]\|_{L^2} \\
         & \les \| \psi_\lambda'(\xi) \tcalF(\phi_L P_{<\lambda/4}^{0}f) \|_{L^2}  +
     \| \psi_\lambda(\xi) \p_\xi \tcalF(\phi_L P_{<\lambda/4}^{0}f) \|_{L^2}.
    \end{aligned}
    \end{equation}
   Next,
   \[
   \tcalF(\phi_L P_{<\lambda/4}^{0}f)(\xi) = \int K_{L}(\xi,\eta) \psi_{<\lambda/4}(\eta)\hat{f}(\eta)\, d\eta
   \]
   where
   \[
   K_{L}(\xi,\eta) = \int e^{-ix(\xi-\eta)} \ol{a(x,\xi)}\phi_L(x)\, dx + \int e^{ix(\xi+\eta)} \ol{b(x,\xi)}\phi_L(x)\, dx.
   \]
   By definition, $\lambda/2<|\xi|<2\lambda$, and $|\eta|\le \lambda/4 $. Hence, $|\xi\pm\eta|\simeq \lambda$ and we conclude that for any $N\ge1$, 
   \[
   \psi_{<\lambda/4}(\eta)|K_{L}(\xi,\eta)|\les_N L\min (1, (L\lambda)^{-N})
   \]
   by integration by parts using Proposition~\ref{prop:dFT} and Lemma~\ref{lem:CalVal}.  Schur's test now finishes the proof of the first term of~\eqref{eq:lamerr}, noting that the $\p_\xi$ in \eqref{eq:xP} costs an extra factor of~$L$ (using also that $\abs{\xi}^{-1}\simeq\lambda^{-1}\les L)$. The second term is handled analogously. In fact, one has 
   \begin{align*}
        \| \langle x\rangle P_{\lambda}^{V} \phi_L P_{>4\lambda}^{0} \|_{L^2 \to L^2} 
       &\le \sum_{j>2} \| \langle x\rangle P_{\lambda}^{V} \phi_L P_{2^j\lambda}^{0} \|_{L^2 \to L^2} \\
       &\les \sum_{j>2}  (1+L)(2^j\,\lambda L)^{-N} \les (1+L)(\lambda L)^{-N}
   \end{align*}
   as claimed. The second bound~\eqref{eq:lamerr2} follows in the same fashion, using Lemma~\ref{lem:derivbounds} to trade a $\p_x$ for a factor of~$\xi$. 
\end{proof}

Lemma~\ref{lem:freqcomp} yields the following frequency truncation statement which will be essential in the next section. 
Throughout, $\chi_I$ stands for a smooth cutoff function to the interval~$I$. 

\begin{lem}\label{lem:freqcomparison}  Let $t\geq 1$, $\lambda \geq t^{-\frac12} $, and $\gamma>\frac12$. Then, 
\begin{align*}
P_{\lambda}^{V}\big(\chi_{|x|\gtrsim t^{\gamma}}f\big) &= P_{\lambda}^V\big(\chi_{|x|\gtrsim t^{\gamma}} P_{[\lambda/4, 4\lambda ]}^0f\big) + g,    
\end{align*}
where the error term $g$ satisfies the bound 
\begin{align*}
    \big\|\langle x\rangle g\big\|_{L^2}+ \big\|\p_x g\big\|_{L^2} &\lesssim_{\gamma, N}t^{-N}\big\|f\big\|_{L^2}.
\end{align*}
Furthermore, we have 
\begin{align}
P_{<t^{-\frac12}}^{V}\big(\chi_{|x|\gtrsim t^{\gamma}}f\big) &= P_{<t^{-\frac12}}^{V}\big(\chi_{|x|\gtrsim t^{\gamma}}P^0_{<4t^{-\frac12}}f\big) + h,   \label{eq:comp2}
\end{align}
with the estimate 
\begin{align*}
    \big\|\langle x\rangle h\big\|_{L^2}+\big\|\p_x h\big\|_{L^2} &\lesssim_{\gamma, N}t^{-N}\big\|f\big\|_{L^2}.
\end{align*}
\end{lem}
\begin{proof} 
This follows from \eqref{eq:lamerr} and \eqref{eq:lamerr2} by summing over dyadic scales of~$L>t^\gamma$. 
\end{proof}

\section{Estimates on the nonlinear source term}
\label{sec:propN}

Here we show how to control the norm 
\[
\big\|J_V(s)\big(|u|^2 u(s)\big)\big\|_{L^2},
\]
by establishing the following result:
\begin{prop}\label{prop:cubic}
For $s\ge1$ one has
\begin{align*}
&\big\|J_V(s) \big(|u|^2 u(s)\big) \big\|_{L^2}
\lesssim W(s)^2\, s^{-1}\big\|\langle J_V(s)\rangle u(s)\big\|_{L^2}, 
\end{align*}
where
\begin{equation}\label{eq:Adef}
W(s): = s^{\frac12}\, \big\|u(s)\big\|_{L^\infty} + s^{-\frac{1}{10}}\big\|J_V(s)u(s)\big\|_{L^2}. 
\end{equation}
\end{prop}
The rest of this section is devoted to the proof of Proposition \ref{prop:cubic}. 
The strategy shall be to decompose
\[
|u|^2 u(x,s) = \chi_{|x|\lesssim s^{\gamma}}|u|^2 u(x,s) + 
\chi_{|x|\gtrsim s^{\gamma}}|u|^2 u(x,s),
\]
 with $0<\gamma - \frac12\ll 1$, and to use the improved local decay estimate Lemma~\ref{lem:Lem12} for the first term on the right. The second term on the right will be more complex to deal with, relying on the algebraic fine structure of the term in conjunction with the frequency comparison Lemma~\ref{lem:freqcomparison} to take advantage of standard Littlewood--Paley techniques. 
We first dispose of the estimate for the inner region $\abs{x}\les s^\gamma$. We will rely on the following simple observation. 

  \begin{lem}\label{lem:crudeJvbound} For all $t\ge0$, one has the estimate 
   \begin{align*}
   \big\|J_V(t) f\big\|_{L^2}\lesssim \big\|\langle x\rangle f\big\|_{L^2} + t \big\|\partial_{x} f\big\|_{L^2} + t \big\|\langle x\rangle^{-1} f\big\|_{L^2}
   \end{align*}
   with some absolute constant.  
   \end{lem}   
 \begin{proof} 
 To begin with, we have the estimate 
 \begin{align*}
 \big\|J_V(t) f\big\|_{L^2} &= \big\|(\partial_{\xi} + 2it\xi)\tilde{f}(\xi)\big\|_{L^2}\leq \big\|\partial_{\xi}\tilde{f}\big\|_{L^2} + 2t\big\|\xi\tilde{f}\big\|_{L^2}. 
 \end{align*}
 The claim then follows from Lemma~\ref{lem:derivbounds}. 
 \end{proof}

Throughout, we will use the notation 
\[
\big\|\langle J_V(s)\rangle u\big\|_{L^2}: = \big\| J_V(s)u\big\|_{L^2} + \big\| u\big\|_{L^2}.
\]

\begin{lem}\label{lem:nonlinearinnerregion}
For $0<\gamma - \frac12\ll 1$ there exists $\delta = \delta(\gamma)>0$ such that for $s\geq 1$
\begin{align*}
\big\|J_V(s)\big(\chi_{|x|\lesssim s^{\gamma}}|u|^2 u(x,s)\big)\big\|_{L_x^2}\lesssim s^{-1-\delta}\, \big\|\langle J_V(s)\rangle u(s)\big\|_{L^2}^3.
\end{align*}
\end{lem}
\begin{proof} 
In light of Lemma~\ref{lem:crudeJvbound},  we have the estimate 
\begin{align*}
\big\|J_V(s)\big(\chi_{|x|\lesssim s^{\gamma}}|u|^2 u(x,s)\big)\big\|_{L_x^2}&\lesssim \big\|\langle x\rangle\big(\chi_{|x|\lesssim s^{\gamma}}|u|^2 u(x,s)\big)\big\|_{L_x^2} + s\, \big\|\partial_x\big(\chi_{|x|\lesssim s^{\gamma}}|u|^2 u(x,s)\big)\big\|_{L_x^2}\\
&\qquad + s\, \big\|\langle x\rangle^{-1}\big(\chi_{|x|\lesssim s^{\gamma}}|u|^2 u(x,s)\big)\big\|_{L_x^2}.
\end{align*}
We estimate each term on the right. 

\medskip

{\it{The estimate for $\big\|\langle x\rangle\big(\chi_{|x|\lesssim s^{\gamma}}|u|^2 u(x,s)\big)\big\|_{L_x^2}$}}. Using H\"older's inequality, we find 
\begin{align*}
\big\|\langle x\rangle\big(\chi_{|x|\lesssim s^{\gamma}}|u|^2 u(x,s)\big)\big\|_{L_x^2}\lesssim \langle s\rangle ^{\gamma}s^{\frac{\gamma}{2}} \big\|u(x,s)\big\|_{L^\infty(|x|\lesssim s^{\gamma})}^3.     
\end{align*}
Then we exploit Lemma~\ref{lem:Lem12} to bound the preceding by 
\begin{align*}
\langle s\rangle ^{\frac32\gamma} \big\|u(x,s)\big\|_{L^\infty(|x|\lesssim s^{\gamma})}^3&\lesssim \langle s\rangle ^{\gamma}s^{\frac{\gamma}{2}} s^{-\frac92(1-\gamma)} \big\|J_V(s)u(s)\big\|_{L^2}^3\lesssim s^{-1} s^{-\frac72 + 6\gamma} \big\|J_V(s)u(s)\big\|_{L^2}.
\end{align*}
We conclude that the desired estimate holds provided $\gamma<\frac{7}{12}$. 

\medskip

{\it{The estimate for $s \big\|\partial_x\big(\chi_{|x|\lesssim s^{\gamma}}|u|^2 u(x,s)\big)\big\|_{L_x^2}$.}} Using the Leibniz rule as well as the triangle and H\"older inequalities, we can estimate
\begin{align*}
s\, \big\|\partial_x\big(\chi_{|x|\lesssim s^{\gamma}}|u|^2 u(x,s)\big)\big\|_{L_x^2}&\lesssim s \,\big\|\partial_x\big(\chi_{|x|\lesssim s^{\gamma}}\big)\big\|_{L_x^2}\cdot \big\||u|^2 u(x,s)\big\|_{L^\infty_x(|x|\lesssim s^{\gamma})}\\
&\qquad  + s\,\big\|\partial_x u\big\|_{L_x^2(|x|\lesssim s^{\gamma})}\cdot\big\||u|^2(x,s)\big\|_{L^\infty_x(|x|\lesssim s^{\gamma})}.
\end{align*}
The first term on the right can be bounded by 
\begin{align*}
s\, \big\|\partial_x\big(\chi_{|x|\lesssim s^{\gamma}}\big)\big\|_{L_x^2} \cdot \big\||u|^2 u(x,s)\big\|_{L^\infty_x(|x|\lesssim s^{\gamma})}&\lesssim s^{1-\frac{\gamma}{2}}s^{-\frac92(1-\gamma)} \big\|J_V(s)u(s)\big\|_{L^2}^3\\
&\lesssim s^{-1}\, s^{-\frac52 + 4\gamma} \,\big\|J_V(s)u(s)\big\|_{L^2}^3.
\end{align*}
This leads to an acceptable bound provided $\gamma<\frac58$. As for the second term on the right, we use that 
\begin{align*}
\partial_x u = \frac{1}{2is} \big(xu - J_0(s)u\big),     
\end{align*}
which implies that 
\begin{align*}
s \,\big\|\partial_x u(s)\big\|_{L_x^2(|x|\lesssim s^{\gamma})}\lesssim \big\||x|u(s)\big\|_{L_x^2(|x|\lesssim s^{\gamma})} + \big\|J_0(s)u(s)\big\|_{L_x^2(|x|\lesssim s^{\gamma})}.  
\end{align*}
Taking advantage of Lemma~\ref{lem:Lem12} together with Corollary~\ref{cor:J_0J_Vbound}, we can further bound this by 
\begin{align*}
\big\||x|u(s)\big\|_{L_x^2(|x|\lesssim s^{\gamma})} + \big\|J_0(s)u(s)\big\|_{L_x^2(|x|\lesssim s^{\gamma})}
\lesssim s^{\frac32\gamma}s^{-\frac32(1-\gamma)}\big\|J_V(s)u(s)\big\|_{L^2} + \big\|\langle J_V(s)\rangle u(s)\big\|_{L^2}.
\end{align*}
One more application of Lemma~\ref{lem:Lem12} then leads to the estimate 
\begin{align*}
&s \,\big\|\partial_x u(x,s)\big\|_{L_x^2(|x|\lesssim s^{\gamma})}\big\||u|^2(x,s)\big\|_{L^\infty_x(|x|\lesssim s^{\gamma})}\\
&\qquad \lesssim \big(s^{\frac32\gamma}s^{-\frac32(1-\gamma)}\,\big\|J_V(s)u(s)\big\|_{L^2} + \big\|\langle J_V(s)\rangle u(s)\big\|_{L^2}\big)\cdot s^{-3(1-\gamma)}\big\|J_V(s)u(s)\big\|_{L^2}^2. 
\end{align*}
It is straightforward to check that this satisfies the bound asserted in the lemma, provided $0<\gamma-\frac12\ll 1$. 

\medskip

{\it{The estimate for $s\, \big\|\langle x\rangle^{-1}\big(\chi_{|x|\lesssim s^{\gamma}}|u|^2 u(x,s)\big)\big\|_{L_x^2}$.}} This is analogous to the preceding case, since $\langle x\rangle^{-1}\chi_{|x|\lesssim s^{\gamma}}$ satisfies similar estimates as $\partial_x\big(\chi_{|x|\lesssim s^{\gamma}}\big)$.
\end{proof}

We next turn our attention to the estimates for the nonlinear term restricted to the outer region $|x|\gtrsim s^{\gamma}$, $\gamma>\frac12$. Here we need to rely on more sophisticated, microlocal arguments. We start with the {\it{small frequency case}}. Throughout, we assume that $s\geq 1$. 

\begin{lem}\label{lem:xlargexismall} 
Let $\gamma>\frac12$. For any $N\ge1$ the following estimate holds
\begin{align*}
 \big\|J_V(s) P^V_{<s^{-\frac12}}\big(\chi_{|x|\gtrsim s^{\gamma}}|u|^2 u(x,s)\big)\big\|_{L_x^2}&\lesssim_{\gamma,N}  \big\|J_0(s)\big(\chi_{|x|\gtrsim s^{\gamma}}P^0_{<4s^{-\frac12}}\big(|u|^2 u(x,s)\big)\big)\big\|_{L_x^2}\\
 &\qquad +s^{-N} \big\|u(s)\big\|_{L^2}\cdot\|u(s)\|_{L^\infty}^2
\end{align*}
uniformly in $s\ge1$.
\end{lem}
\begin{proof} 
Fix some  $\gamma > \frac12$. In view of Lemma~\ref{lem:freqcomparison}, one has 
\begin{align*}
P^V_{<s^{-\frac12}}\big(\chi_{|x|\gtrsim s^{\gamma}}|u|^2 u(x,s)\big) &= P^V_{<s^{-\frac12}}\big(\chi_{|x|\gtrsim s^{\gamma}}P^0_{<4s^{-\frac12}}\big(|u|^2 u(x,s)\big)\big) + h(x,s), 
\end{align*}
where by Lemma~\ref{lem:crudeJvbound} 
\begin{align*}
\big\|J_V(s) h(s)\big\|_{L^2}\lesssim_N s^{-N} \big\|u(s)\big\|_{L^2}\cdot\|u(s)\|_{L^\infty}^2.     
\end{align*}
Next, we commute $J_V(s)$
\begin{align*}
J_V(s)P^V_{<s^{-\frac12}}\big(\chi_{|x|\gtrsim s^{\gamma}}P^0_{<4s^{-\frac12}}\big(|u|^2 u(x,s)\big)\big)&=\big[J_V(s),P^V_{<s^{-\frac12}}\big]\big(\chi_{|x|\gtrsim s^{\gamma}}P^0_{<4s^{-\frac12}}\big(|u|^2 u(x,s)\big)\big)\\
&\qquad + P^V_{<s^{-\frac12}}J_V(s)\big(\chi_{|x|\gtrsim s^{\gamma}}P^0_{<4s^{-\frac12}}\big(|u|^2 u(x,s)\big)\big)\\
&=: I + II. 
\end{align*}
{\emph{The estimate for $II$.}} Using the $L^2$-boundedness of the projection operator $P^V_{<s^{-\frac12}}$, it suffices to bound 
\begin{align*}
\big\|J_V(s)\big(\chi_{|x|\gtrsim s^{\gamma}}P^0_{<4s^{-\frac12}}\big(|u|^2 u(x,s)\big)\big)\big\|_{L_x^2}. 
\end{align*}
Owing to Lemma~\ref{lem:crudeJvbound}, we can bound the preceding term by 
\begin{align*}
\big\|J_V(s)\big(\chi_{|x|\gtrsim s^{\gamma}}P^0_{<4s^{-\frac12}}\big(|u|^2 u(x,s)\big)\big)\big\|_{L_x^2}&\lesssim \big\|\langle x\rangle\big(\chi_{|x|\gtrsim s^{\gamma}}P^0_{<4s^{-\frac12}}\big(|u|^2 u(x,s)\big)\big)\big\|_{L_x^2}\\
&\qquad + s\,\big\|\partial_x\big(\chi_{|x|\gtrsim s^{\gamma}}P^0_{<4s^{-\frac12}}\big(|u|^2 u(x,s)\big)\big)\big\|_{L_x^2}\\
&\qquad + s\,\big\|\langle x\rangle^{-1}\big(\chi_{|x|\gtrsim s^{\gamma}}P^0_{<4s^{-\frac12}}\big(|u|^2 u(x,s)\big)\big)\big\|_{L_x^2}.
\end{align*}
Next, we compute 
\begin{align*}
 s\,\big\|\partial_x\big(\chi_{|x|\gtrsim s^{\gamma}}P^0_{<4s^{-\frac12}}\big(|u|^2 u(x,s)\big)\big)\big\|_{L_x^2}&\lesssim s^{\frac12}\, \big\|\chi_{|x|\gtrsim s^{\gamma}}P^0_{<4s^{-\frac12}}\big(|u|^2 u(x,s)\big)\big\|_{L_x^2}\\
 &\lesssim \big\|x \,\chi_{|x|\gtrsim s^{\gamma}}P^0_{<4s^{-\frac12}}\big(|u|^2 u(x,s)\big)\big\|_{L_x^2}.
\end{align*}
Here we also used that $s^{\frac12}\lesssim |x|$ on the support of the expression. 
Similarly, we find that 
\begin{align*}
& s\,\big\|\langle x\rangle^{-1}\big(\chi_{|x|\gtrsim s^{\gamma}}P^0_{<4s^{-\frac12}}\big(|u|^2 u(x,s)\big)\big)\big\|_{L_x^2}\lesssim \big\|x\,\chi_{|x|\gtrsim s^{\gamma}}P^0_{<4s^{-\frac12}}\big(|u|^2 u(x,s)\big)\big\|_{L_x^2}.
\end{align*}
The assertion of the lemma for the term $II$ will then be a consequence of the bound 
\begin{align*}
 \big\|\langle x\rangle\big(\chi_{|x|\gtrsim s^{\gamma}}P^0_{<4s^{-\frac12}}\big(|u|^2 u(x,s)\big)\big)\big\|_{L_x^2}&\lesssim\big\|J_0(s)\big(\chi_{|x|\gtrsim s^{\gamma}}P^0_{<4s^{-\frac12}}\big(|u|^2 u(x,s)\big)\big)\big\|_{L_x^2}\\
 &\qquad+ C_Ns^{-N}\, \big\|u(s)\big\|_{L^2}\cdot\|u(x,s)\|_{L^\infty}^2.
\end{align*}
To see this, observe that on the one hand
\begin{align*}
&\big\|\langle x\rangle\big(\chi_{|x|\gtrsim s^{\gamma}}P^0_{<4s^{-\frac12}}\big(|u|^2 u(x,s)\big)\big)\big\|_{L_x^2}\lesssim \big\| x \, \chi_{|x|\gtrsim s^{\gamma}}P^0_{<4s^{-\frac12}}\big(|u|^2 u(x,s)\big)\big\|_{L_x^2}.
\end{align*}
On the other hand, we have 
\begin{align*}
J_0(s)\big(\chi_{|x|\gtrsim s^{\gamma}}P^0_{<4s^{-\frac12}}\big(|u|^2 u(x,s)\big)\big)& = x\big(\chi_{|x|\gtrsim s^{\gamma}}P^0_{<4s^{-\frac12}}\big(|u|^2 u(x,s)\big)\big)\\
&\qquad - 2is\,\partial_x P^0_{<20s^{-\frac12}}\big(\chi_{|x|\gtrsim s^{\gamma}}P^0_{<4s^{-\frac12}}\big(|u|^2 u(x,s)\big)\big) + h_2(x,s), 
\end{align*}
where 
\begin{align*}
 \big\|{h_2(s)}\big\|_{L^2}\lesssim_N s^{-N}\, \big\|u(s)\big\|_{L^2}\|u(s)\|_{L^\infty}^2.    
\end{align*}
Therefore, upon choosing a sufficiently large implicit constant for $\abs{x}\gtrsim s^\gamma$
\begin{align*}
\big\|2is\partial_x P^0_{<20s^{-\frac12}}\big(\chi_{|x|\gtrsim s^{\gamma}}P^0_{<4s^{-\frac12}}\big(|u|^2 u(x,s)\big)\big)\big\|_{L_x^2}&\les s^{\frac12}\, \big\|\chi_{|x|\gtrsim s^{\gamma}}P^0_{<4s^{-\frac12}}\big(|u|^2 u(x,s)\big)\big\|_{L_x^2}\\
&\ll\big\|x\, \chi_{|x|\gtrsim s^{\gamma}}P^0_{<4s^{-\frac12}}\big(|u|^2 u(x,s)\big) \big\|_{L_x^2}.
\end{align*}
It follows that
\begin{equation}\label{eq:bdfortermII}
\begin{aligned}
\big\||x|\big(\chi_{|x|\gtrsim s^{\gamma}}P^0_{<4s^{-\frac12}}\big(|u|^2 u(x,s)\big)\big)\big\|_{L_x^2}&\lesssim \big\|J_0(s)\big(\chi_{|x|\gtrsim s^{\gamma}}P^0_{<4s^{-\frac12}}\big(|u|^2 u(x,s)\big)\big)\big\|_{L_x^2}\\
&\qquad + C_Ns^{-N}\, \big\|u(s)\big\|_{L^2}\|u(s)\|_{L^\infty}^2.
\end{aligned}
\end{equation}
As observed before, this implies the lemma for term $II$. 

\medskip

\noindent\emph{The estimate for $I$.} The argument for the term $I$ is similar. In fact, by \eqref{eq:bdfortermII} we can bound 
\begin{align*}
\big\|\chi_{|x|\gtrsim s^{\gamma}}P^0_{<4s^{-\frac12}}\big(|u|^2 u(x,s)\big)\big\|_{L_x^2}
&\lesssim s^{-\gamma}\, \big\|J_0(s)\big(\chi_{|x|\gtrsim s^{\gamma}}P^0_{<4s^{-\frac12}}\big(|u|^2 u(x,s)\big)\big)\big\|_{L_x^2} \\
&\qquad\qquad+ C_Ns^{-N}\, \big\|u(s)\big\|_{L_x^2}\|u(s)\|_{L^\infty}^2.
\end{align*} 
Since the commutator
\[
\big[J_V(s),P^V_{<s^{-\frac12}}\big]
\]
is given by a (distorted) Fourier multiplier bounded by $s^{\frac12}$, we obtain
\begin{align}
&\big\|\big[J_V(s),P^V_{<s^{-\frac12}}\big]\big(\chi_{|x|\gtrsim s^{\gamma}}P^0_{<4s^{-\frac12}}\big(|u|^2 u(x,s)\big)\big)\big\|_{L_x^2}\label{eq:JVcomm}\\
&\qquad \lesssim \big\|J_0(s)\big(\chi_{|x|\gtrsim s^{\gamma}}P^0_{<4s^{-\frac12}}\big(|u|^2 u(x,s)\big)\big)\big\|_{L_x^2} + C_Ns^{-N}\, \big\|u(s)\big\|_{L^2}\|u(s)\|_{L^\infty}^2\nn
\end{align}
which concludes the proof.
\end{proof}

The preceding lemma gives the transference of  $J_V(s)$  to $J_0(s)$ when acting on the suitably localized nonlinear term. Our next task is to give bounds on the $J_0(s)$ term, relying on the microlocal structure. The advantage of $J_0(s)$ is its local nature, which results in a product rule as this operator acts on the cubic nonlinearity. 

\begin{lem}\label{lem:J_0nonlinear1} 
For $\gamma>\frac12$, but close to $\frac12$,  and $N\ge1$, one has for all $s\ge1$
\begin{align*}
&\big\|J_0(s)\big(\chi_{|x|\gtrsim s^{\gamma}}P^0_{<s^{-\frac12}}\big(|u|^2 u(x,s)\big)\big)\big\|_{L_x^2}
\lesssim_\gamma  W(s)^2\, s^{-1}\big\|\langle J_V(s)\rangle u(s)\big\|_{L_x^2},
\end{align*}
where we recall from \eqref{eq:Adef} that
\begin{equation*}
W(s): = s^{\frac12}\, \big\|u(s)\big\|_{L^\infty} + s^{-\frac{1}{10}}\big\|J_V(s)u(s)\big\|_{L^2}. 
\end{equation*}
\end{lem}
\begin{proof} 
The proof hinges on the well-known product rule
\begin{equation}\label{eq:nullform}
J_0(s)\big(|u|^2 u(s)\big) = 2J_0(s)u(s)\,|u(s)|^2 - u^2(s)\overline{J_0(s) u(s)}.
\end{equation}
Write 
\begin{align*}
J_0(s)\big(\chi_{|x|\gtrsim s^{\gamma}}P^0_{<s^{-\frac12}}\big(|u|^2 u(x,s)\big)\big) &= \big[J_0(s), \chi_{|x|\gtrsim s^{\gamma}}P^0_{<s^{-\frac12}}]\big(|u|^2 u(x,s)\big)\\
&\qquad+ \chi_{|x|\gtrsim s^{\gamma}}P^0_{<s^{-\frac12}}J_0(s)\big(|u|^2 u(x,s)\big)\\
&=:I+II.
\end{align*}

\noindent\emph{The estimate of $I$.} 
Further decompose 
\begin{align*}
\big[J_0(s), \chi_{|x|\gtrsim s^{\gamma}}P^0_{<s^{-\frac12}}] = [J_0(s), \chi_{|x|\gtrsim s^{\gamma}}]P^0_{<s^{-\frac12}} + \chi_{|x|\gtrsim s^{\gamma}}\big[J_0(s), P^0_{<s^{-\frac12}}].
\end{align*}
For the contribution of the first term on the right, we have 
\begin{align*}
&\big\|[J_0(s), \chi_{|x|\gtrsim s^{\gamma}}]P^0_{<s^{-\frac12}}\big(|u|^2 u(x,s)\big)\big\|_{L_x^2}\lesssim s\, \big\|\partial_x\big(\chi_{|x|\gtrsim s^{\gamma}}\big)P^0_{<s^{-\frac12}}\big(|u|^2 u(x,s)\big)\big\|_{L_x^2}.
\end{align*}
To estimate this term, we note that we may restrict the product $|u|^2 u(x,s)$ to $|x|\lesssim s^{\gamma}$ (with a large implicit constant, compared to the first spatial cutoff) up to a term rapidly decaying in $s$. Recalling Lemma~\ref{lem:Lem12}, we find that 
\begin{align*}
 &s\, \big\|\partial_x\big(\chi_{|x|\gtrsim s^{\gamma}}\big)P^0_{<s^{-\frac12}}\big(|u|^2 u(x,s)\big)\big\|_{L_x^2}\lesssim s^{1-\frac{\gamma}{2}}s^{-\frac92(1-\gamma)}\, \big\|J_V(s)u(s)\big\|_{L_x^2}^3 + C_N s^{-N}\big\|u(s)\big\|_{L^2}\|u(s)\|_{L^\infty}^2. 
\end{align*}
Since $ 0<\gamma - \frac12\ll 1$, this bound implies the one claimed in the lemma for this contribution. 

Next, we note that $\big[J_0(s), P^0_{<s^{-\frac12}}]=[x,P^0_{<s^{-\frac12}}]$ is given by a smooth (standard) Fourier multiplier bounded by $s^{\frac12}$, and that we have 
\begin{align*}
\chi_{|x|\gtrsim s^{\gamma}}\big[J_0(s), P^0_{<s^{-\frac12}}]
=\chi_{|x|\gtrsim s^{\gamma}}\big[J_0(s), P^0_{<s^{-\frac12}}]\chi_{|x|\gtrsim s^{\gamma}} + E,
\end{align*}
where $\|E\|_{L_x^2\rightarrow L_x^2}\lesssim_N s^{-N}$ (the implicit constant in the second spatial cutoff being much smaller than in the first). To bound the contribution of the first operator on the right, when applied to $|u|^2 u(s)$, we introduce an operator $T(s)$ via 
\begin{align*}
&\chi_{|x|\gtrsim s^{\gamma}}\big[J_0(s), P^0_{<s^{-\frac12}}]\chi_{|x|\gtrsim s^{\gamma}}=\chi_{|x|\gtrsim s^{\gamma}}\big[J_0(s), P^0_{<s^{-\frac12}}]T(s)\circ T(s)^{-1}\chi_{|x|\gtrsim s^{\gamma}}
\end{align*}
where we define, with some small $\delta>0$,
\begin{align*}
T(s) &= \big(1 - 2ix^{-1}s\chi_{|x|\gtrsim s^{\gamma-\delta}}P^0_{<s^{-\frac12+\delta}}\partial_x\big)^{-1}= \sum_{\ell\geq 0}\big(2ix^{-1}s \chi_{|x|\gtrsim s^{\gamma-\delta}}P^0_{<s^{-\frac12+\delta}}\partial_x\big)^\ell.
\end{align*}
The sum here converges absolutely with respect to the operator norm $\|\cdot\|_{L^2\rightarrow L^2}$, provided the implicit constant in the spatial cutoff is large. In fact, assuming that $0<2\delta<\gamma-\frac12$
\[
\| \big(2ix^{-1}s \chi_{|x|\gtrsim s^{\gamma-\delta}}P^0_{<s^{-\frac12+\delta}}\partial_x\big)^\ell\|_{L^2 \to L^2}\les s^{\ell(\frac12+2\delta-\gamma)}, 
\]
which implies that the tail of the sequence is bounded by~$O(s^{-N})$ for any $N$. 
Since the commutator $\big[J_0(s), P^0_{<s^{-\frac12}}]$ localizes to (standard) frequency $\simeq s^{-\frac12}$, we infer that 
\begin{align*}
&\chi_{|x|\gtrsim s^{\gamma}}\big[J_0(s), P^0_{<s^{-\frac12}}]T(s)\circ T(s)^{-1}\chi_{|x|\gtrsim s^{\gamma}}= \chi_{|x|\gtrsim s^{\gamma}}\big[J_0(s), P^0_{<s^{-\frac12}}]T(s)\circ x^{-1} J_0(s)\chi_{|x|\gtrsim s^{\gamma}} + \tilde{E}, 
\end{align*}
where $\|\tilde{E}\|_{L_x^2\rightarrow L_x^2}\lesssim_N s^{-N}$. Note that by the remark concerning the tails of the Neumann series for $T(s)$, only finitely many terms in the series matter for this commutator argument. 
We now control the effect of the first operator on the right on the nonlinear term 
$
|u|^2 u(s). 
$
Observing that
\[
J_0(s)\chi_{|x|\gtrsim s^{\gamma}} = \big[J_0(s),\chi_{|x|\gtrsim s^{\gamma}}\big] + \chi_{|x|\gtrsim s^{\gamma}}J_0(s), 
\]
we first estimate 
\begin{align*}
&\big\|\chi_{|x|\gtrsim s^{\gamma}}\big[J_0(s), P^0_{<s^{-\frac12}}\big]T(s)\circ x^{-1} [J_0(s),\chi_{|x|\gtrsim s^{\gamma}}]\big(|u|^2 u(x,s)\big)\big\|_{L_x^2}\\
&\qquad\lesssim \big\|\big[J_0(s), P^0_{<s^{-\frac12}}\big]\circ T(s)\big\|_{L^2\rightarrow L^2}\, s^{-\gamma}\, \big\|[J_0(s),\chi_{|x|\gtrsim s^{\gamma}}]\big(|u|^2 u(x,s)\big)\big\|_{L_x^2}.
\end{align*}
The factor $s^{-\gamma}$ is a consequence of the factor $x^{-1}$ as well as the localization of the term in the first line. The desired bound for this contribution is then a consequence of the estimate 
\[
\big\|\big[J_0(s), P^0_{<s^{-\frac12}}\big]\circ T(s)\big\|_{L^2\rightarrow L^2}\, s^{-\gamma}\lesssim 1, 
\]
as well as the strong local decay estimate to bound 
\[
\big\|[J_0(s),\chi_{|x|\gtrsim s^{\gamma}}]\big(|u|^2 u(x,s)\big)\big\|_{L_x^2},
\]
which we did at the beginning of this proof (see also the proof of Lemma \ref{lem:nonlinearinnerregion}). 

\medskip
We are now reduced to bounding the term 
\begin{align*}
\big\|\chi_{|x|\gtrsim s^{\gamma}}\big[J_0(s), P^0_{<s^{-\frac12}}\big]T(s)\circ x^{-1} \chi_{|x|\gtrsim s^{\gamma}} J_0(s)\big(|u|^2 u(x,s)\big)\big\|_{L_x^2}.
\end{align*}
Using \eqref{eq:nullform}, H\"older's inequality, and  Corollary~\ref{cor:J_0J_Vbound}, we obtain the bound 
\[
\big\|J_0(s)\big(|u|^2 u(x,s)\big)\big\|_{L_x^2}\lesssim W(s)^2\, \big\|\langle J_V(s)\rangle u(s)\big\|_{L^2}. 
\]
The desired estimate is then a consequence of the operator bound 
\begin{align*}
&\big\|\chi_{|x|\gtrsim s^{\gamma}}\big[J_0(s), P^0_{<s^{-\frac12}}\big]T(s)\circ x^{-1} \chi_{|x|\gtrsim s^{\gamma}}\big\|_{L^2\rightarrow L^2}\lesssim \big\|\big[J_0(s), P^0_{<s^{-\frac12}}\big]\circ T(s)\big\|_{L^2\rightarrow L^2}\, s^{-\gamma}\lesssim 1.
\end{align*}

\noindent\emph{The estimate of $II$.} 
It only remains to bound the term $$\chi_{|x|\gtrsim s^{\gamma}}P^0_{<s^{-\frac12}}J_0(s)\big(|u|^2 u(x,s)\big),$$ which follows from the above bound for $\big\|J_0(s)\big(|u|^2 u(s)\big)\big\|_{L^2}$, as well as the fact that $\chi_{|x|\gtrsim s^{\gamma}}P^0_{<s^{-\frac12}}$ acts in a bounded fashion on $L^2$. 
\end{proof}

The preceding two Lemmas~\ref{lem:xlargexismall} and~\ref{lem:J_0nonlinear1} give us control over the {\it{low frequency term}}
\[
P^V_{<s^{-\frac12}}\big(\chi_{|x|\gtrsim s^{\gamma}}|u|^2 u(x,s)\big).
\]
It now remains to deal with the more difficult {\it{high frequency term}}
\[
P^V_{>s^{-\frac12}}\big(\chi_{|x|\gtrsim s^{\gamma}}|u|^2 u(x,s)\big).
\]
Here, Lemma~\ref{lem:directJVJO} as well as Lemma~\ref{lem:microdec} turn out to be important, in addition to Lemma~\ref{lem:freqcomparison}. 

\begin{lem}\label{lem:J_Vnonlinear2} 
For $\gamma>\frac12$ and $s\ge1$
\begin{align*}
&\big\|J_V(s)P^V_{>s^{-\frac12}}\big(\chi_{|x|\gtrsim s^{\gamma}}\big(|u|^2 u(x,s)\big)\big)\big\|_{L_x^2}
\lesssim W(s)^2\, s^{-1}\big\|\langle J_V(s)\rangle u(s)\big\|_{L^2},
\end{align*}
where $W(s)$ is as defined in \eqref{eq:Adef}.
\end{lem}

\begin{proof}
Applying the operator $J_V$, we write 
\begin{align*}
&J_V(s)P^V_{>s^{-\frac12}}\big(\chi_{|x|\gtrsim s^{\gamma}}|u|^2 u(x,s)\big)\\
& \qquad = [J_V(s), P^V_{>s^{-\frac12}}]\big(\chi_{|x|\gtrsim s^{\gamma}}|u|^2 u(x,s)\big)+P^V_{>s^{-\frac12}}J_V(s)\big(\chi_{|x|\gtrsim s^{\gamma}}|u|^2 u(x,s)\big).
\end{align*}
The first term is treated by arguments that have already appeared above, so we only sketch it. On the one hand,  
\[
[J_V(s), P^V_{>s^{-\frac12}}]f = \widetilde\calF^{*} \,(s^{\frac12} \psi'( s^{\frac12} \xi) \tilde f(\xi)),
\]
which localizes to distorted frequency $\simeq s^{-\frac12}$. By Lemma~\ref{lem:freqcomp}, this restricts us to a comparable range relative to the standard Fourier transform. Thus,
\begin{align*}
\|[J_V(s), P^V_{>s^{-\frac12}}]\big(\chi_{|x|\gtrsim s^{\gamma}}|u|^2 u(x,s)\big)\|_{L^2} &\les s^{\frac12} \|\chi_{|x|\gtrsim s^{\gamma}}  \big(P^0_{\simeq s^{-\frac12}}\, \big(|u|^2 u(x,s)\big)\big) \|_{L^2}\\
&\qquad + C_N s^{-N}\big\|u(s)\big\|_{L^2}\|u(s)\|_{L^\infty}^2. 
\end{align*}
Next, cf.~\eqref{eq:JVcomm}, 
\begin{align*}
  \|\chi_{|x|\gtrsim s^{\gamma}}  \big(P^0_{\simeq s^{-\frac12}}\, \big(|u|^2 u(x,s) \big)\big)\|_{L^2} &\les  s^{-\gamma} \big\|\chi_{|x|\gtrsim s^{\gamma}}\big(  P^0_{\simeq s^{-\frac12}}\,\big(J_0(s) \big( |u|^2 u(x,s) \big)\big)\big\|_{L^2}\\
  &\qquad +C_N s^{-N}\big\|u(s)\big\|_{L^2}\|u(s)\|_{L^\infty}^2.
\end{align*}
The mechanism here being that $x$ in $J_0(s)$ dominates $s\p_x$ due to the microlocal support. To finish, apply~\eqref{eq:nullform}
and Lemma~\ref{cor:J_0J_Vbound}. We now turn to the  estimate of the second term on the right. 
Taking advantage of Lemma~\ref{lem:directJVJO}, we decompose
\[
J_V(s) = T_0 J_0(s) + \sum_{j=1,2} T_j(s) f, 
\]
where we set 
\begin{align*}
&T_1(s)f: = s\, \widetilde{\mathcal{F}}^*\Big(\int_{-\infty}^{\infty} \big(K_1(x,\xi)e^{ix\xi} + K_2(x,\xi) e^{-ix\xi}\big)f(x)\,dx\Big),\\
&T_2(s)f: = \widetilde{\mathcal{F}}^*\Big(\int_{-\infty}^{\infty} \big(K_3(x,\xi)e^{ix\xi} + K_4(x,\xi) e^{-ix\xi}\big)f(x)\,dx\Big).
\end{align*}
The contribution of the operator $T_0 J_0(t)$ was essentially already handled in the proof of Lemma~\ref{lem:J_0nonlinear1}. In fact, upon observing that a version of Lemma \ref{lem:freqcomparison} applies here, we only need to add the observation that $P^V_{>s^{-\frac12}}\circ T_0$ acts as a bounded operator on $L^2$.

This reduces things to bounding 
\begin{align*}
P^V_{>s^{-\frac12}}T_j(s)\big(\chi_{|x|\gtrsim s^{\gamma}}|u|^2 u(x,s)\big),\quad j = 1, 2.     
\end{align*}
The key shall be to decompose the term into two contributions, namely a (standard) paraproduct and an error term,
\begin{equation}\label{eq:keydecomposition}\begin{split}
P^V_{>s^{-\frac12}}T_j(s)\big(\chi_{|x|\gtrsim s^{\gamma}}|u|^2 u(x,s)\big) &= P^V_{>s^{-\frac12}}T_j(s)\big(\chi_{|x|\gtrsim s^{\gamma}}P^0_{<C^{-1}s^{-\frac12}}(|u|^2) u(x,s)\big)\\
&\qquad  + P^V_{>s^{-\frac12}}T_j(s)\big(\chi_{|x|\gtrsim s^{\gamma}}P^0_{\geq C^{-1}s^{-\frac12}}(|u|^2) u(x,s)\big),
\end{split}\end{equation}
for some $C\gg 1$. 
We estimate these separately.

\medskip

{\it{(1) The estimate for the paraproduct}}. Let us localize the (distorted) output frequency of the term by replacing $P^V_{>s^{-\frac12}}$ by $P^V_{\lambda}$ for dyadic $\lambda>s^{-\frac12}$. 
Using Lemma~\ref{lem:freqcomparison} and recalling the definition of the $T_j$, we see (via a version of Lemma \ref{lem:freqcomparison}) that we can replace this term by 
\begin{align*}
P^V_{\lambda}T_j(s)\big(\chi_{|x|\gtrsim s^{\gamma}}P^0_{<C^{-1}s^{-\frac12}}(|u|^2) P_{[\lambda s^{-\delta}, \lambda s^{\delta}]}^Vu(x,s)\big),\qquad j = 1, 2, 
\end{align*}
up to an error term decaying rapidly with respect to $s$. Throughout we recall that $\gamma - \frac12\gg \delta>0$ by choice. This in particular implies that 
\[
s^{\gamma-\delta}\, \lambda\geq s^{\gamma - \frac12-\delta}
\]
is a positive power of $s$.

\smallskip

{\it{(1a) The estimate for $j = 1$.}} Recalling the estimate for the kernels of $T_1$ from Lemma~\ref{lem:directJVJO}, invoking Lemma~\ref{lem:microdec}, and finally recalling Lemma~\ref{lem:CalVal}, we obtain 
\begin{align*}
&\big\|P^V_{\lambda}T_j(s)\big(\chi_{|x|\gtrsim s^{\gamma}}P^0_{<C^{-1}s^{-\frac12}}(|u|^2) P_{[\lambda s^{-\delta}, \lambda s^{\delta}]}^Vu(x,s)\big)\big\|_{L_x^2}\\
&\qquad \lesssim s^{-\alpha}\big\|P^0_{<C^{-1}s^{-\frac12}}(|u|^2)\big\|_{L_x^\infty}\, s\big\|\langle x\rangle^{-\frac12-\delta}\langle x\lambda\rangle^{-\frac12}P_{[\lambda s^{-\delta}, \lambda s^{\delta}]}^Vu(x,s)\big\|_{L_x^\infty}\\
&\qquad \lesssim s^{-\alpha-1}\, W(s)^2\, \big\|\langle J_V(s)\rangle P_{[\lambda s^{-\delta}, \lambda s^{\delta}]}^Vu(s)\big\|_{L^2},
\end{align*}
where $\alpha = \frac12(\gamma - \frac12-3\delta)$. Note that the factor $s$ in the second line comes from the definition of $T_1$. 
The non-localized estimate follows by square summing over dyadic $\lambda>s^{-\frac12}$ and exploiting approximate orthogonality of the $P_{[\lambda s^{-\delta}, \lambda s^{\delta}]}^Vu(x,s)$, at the expense of replacing $\alpha$ by $\alpha-$. Thus this results in a bound that is better than required. 

\smallskip

{\it{(1b) The estimate for $j = 2$.}} This is similar to the case (1a), except that there is no loss of $s$, and due to the kernel bounds in Lemma~\ref{lem:directJVJO} the factor 
\[
\big\|\langle x\rangle^{-\frac12-\delta}\langle x\lambda\rangle^{-\frac12}P_{[\lambda s^{-\delta}, \lambda s^{\delta}]}^Vu(x,s)\big\|_{L_x^\infty}
\]
gets replaced by 
\[
\big\|(1+\lambda^{-1}\langle x\lambda\rangle^{-\frac12})P_{[\lambda s^{-\delta}, \lambda s^{\delta}]}^Vu(x,s)\big\|_{L_x^2}.
\]
{\it{(2) The estimate for the remainder term in \eqref{eq:keydecomposition}.}} The argument here is similar to one used in~\cite{KS}. 
We use the simple identity (here the operator $\partial_x^{-1}$ is defined via passage to the Fourier side)
\begin{equation}
    P^0_{\geq C^{-1}s^{-\frac12}}(|u(s)|^2) = (2is\partial_{x})^{-1}P^0_{\geq C^{-1}s^{-\frac12}}\big(u(s)\overline{J_0(s)u(s)} - J_0(s)u(s) \overline{u(s)}\big). 
\label{eq:KSmfld}
\end{equation}
In view of the operator bound
\[
\big\|(2is\partial_{x})^{-1}P^0_{\geq C^{-1}s^{-\frac12}}\big\|_{L^2\rightarrow L^2}\lesssim s^{-\frac12},
\]
we obtain the estimate 
\begin{align*}
\big\|P^0_{\geq C^{-1}s^{-\frac12}}(|u(s)|^2)\big\|_{L_x^2}\lesssim s^{-\frac12}\, \big\|J_0(s)u(s)\big\|_{L^2}\, \big\|u(s)\big\|_{L^\infty}\lesssim s^{-1}\, W(s)\|\langle J_V(s) \rangle u(s)\|_{L^2}.    
\end{align*}
Here we have used H\"older's inequality, as well as Corollary~\ref{cor:J_0J_Vbound}. If we then invoke Lemma~\ref{lem:directJVJO} together with the localization $|x|\gtrsim s^{\gamma}$, and also recall Lemma~\ref{lem:CalVal}, we deduce that  
\begin{align*}
\big\| P^V_{>s^{-\frac12}}T_1(s)\big(\chi_{|x|\gtrsim s^{\gamma}}P^0_{\geq C^{-1}s^{-\frac12}}(|u|^2) u(x,s)\big)\big\|_{L_x^2}&\lesssim s^{1-\gamma}\, \big\|\chi_{|x|\gtrsim s^{\gamma}}P^0_{\geq C^{-1}s^{-\frac12}}(|u|^2) u(x,s)\big\|_{L_x^2}\\
&\lesssim s^{1-\gamma}\, \big\|\chi_{|x|\gtrsim s^{\gamma}}P^0_{\geq C^{-1}s^{-\frac12}}(|u(x,s)|^2)\big\|_{L_x^2}\, \big\| u(s)\big\|_{L^\infty}.
\end{align*}
Finally invoking the bound from before for the first $\|\cdot\|_{L_x^2}$-norm, we obtain the estimate 
\begin{align*}
&s^{1-\gamma}\, \big\|\chi_{|x|\gtrsim s^{\gamma}}P^0_{\geq C^{-1}s^{-\frac12}}(|u(x,s)|^2)\big\|_{L_x^2}\, \big\| u(s)\big\|_{L^\infty}\lesssim s^{\frac12-\gamma}\, s^{-1}\, W(s)^2\|\langle J_V (s)\rangle u(s)\|_{L^2}.
\end{align*}
This is better than required since $\gamma>\frac12$. 
The estimate for the term involving $T_2$ is similar, one merely needs to replace $s^{\frac12-\gamma}$ by $1$ at the end. 
\end{proof}

\section{Leading order asymptotics via wave packets}

In this section we follow Ifrim and Tataru's wave packet approach~\cite{IT}, as modified by Stewart \cite{Gavin} to accommodate a potential,  in order to derive the leading order asymptotics as well as a global $L^\infty$-bound. Together with the preceding nonlinear estimates, this will enable us to establish global existence of small solutions. As in~\cite{IT} we define the wave packet
\begin{equation}\label{eq:wavepacket}
\Psi_v(x, t): = e^{-i\frac{x^2}{4t}}\chi\big(\frac{x-vt}{\sqrt{t}}\big),\quad t\geq 1. 
\end{equation}
Here, $\chi\in C_0^\infty(\mathbb{R})$ is an even, normalized bump function, i.e.,  $\int \chi\, dx = 1$, with $\textnormal{supp}(\chi)\subset (-5,5)$.
Furthermore, for a given function $u$ we define its {\it{asymptotic profile}}
\begin{equation}\label{eq:profile}
\alpha(v,t) = \int u(x, t)\overline{\Psi_v(x,t)}\,dx. 
\end{equation}
The next lemma shows that for long term dynamics, $\alpha$ is a good substitute for $u$, in a version adapted to our setting:
\begin{lem}\label{lem:mainprofile} We have the bounds
\begin{equation}\label{eq:alpha_simplebd}
\big\|\alpha(v,t)\big\|_{L_v^\infty}\leq t^{\frac12}\big\|u(x,t)\big\|_{L_x^\infty},\quad \norm{\alpha(v,t)}_{L^2_v}\leq \norm{u(x,t)}_{L^2_x}.
\end{equation}
Furthermore, there holds that
\begin{equation}\label{eq:alpha_compare}
\begin{aligned}
\big\|u(vt, t) - t^{-\frac12}e^{-i\frac{tv^2}{4}}\alpha(v,t)\big\|_{L_v^\infty}&\lesssim t^{-\frac34}\, \big\|\langle J_V(t)\rangle u(\cdot,t)\big\|_{L^2},\\
\big\|u(vt, t) - t^{-\frac12}e^{-i\frac{tv^2}{4}}\alpha(v,t)\big\|_{L_v^2}&\lesssim t^{-1}\, \big\|\langle J_V(t)\rangle u(\cdot,t)\big\|_{L^2}.
\end{aligned}
\end{equation}
\end{lem}
The proof is a direct consequence of Corollary \ref{cor:J_0J_Vbound} and \cite[Lemma 2.2]{IT}, we give it here for completeness.
\begin{proof}
 Setting 
\[
q(x,t) = e^{i\frac{x^2}{4t}}u(x,t), 
\]
we find that 
\[
t^{-\frac12}\alpha(v, t) = \big[q(t\cdot, t)*t^{\frac12}\chi(t^{\frac12}\cdot)\big](v),
\]
which implies the first claim \eqref{eq:alpha_simplebd}.

Towards \eqref{eq:alpha_compare}, by the normalisation of $\chi$ we can write 
\begin{align*}
q(vt, t) - t^{-\frac12}\alpha(v,t) = \int\big(q(vt, t) - q((v-z)t, t)\big)t^{\frac12}\chi(t^{\frac12}z)\,dz. 
\end{align*}
Using the fundamental theorem of calculus, we can express the difference
\begin{equation}\label{eq:qdiff-expand}
q(vt, t) - q((v-z)t, t) = z\int_0^1 t\partial_xq\big((v-hz)t, t\big)\,dh. 
\end{equation}
In order to control the integral, we use the relation 
\[
-2it\partial_xq(x, t) = e^{i\frac{x^2}{4t}}J_0(t)u(x, t).
\]
It follows that
\begin{equation}
  q(vt, t) - t^{-\frac12}\alpha(v,t)=\frac{i}{2}\int_0^1\int zt^{\frac12}\chi(t^{\frac12}z) e^{i\frac{((v-hz)t)^2}{4t}}J_0(t)u((v-hz)t, t) dzdh,  
\end{equation}
and thus by Corollary \ref{cor:J_0J_Vbound}
\begin{equation}
\begin{aligned}
  \abs{q(vt, t) - t^{-\frac12}\alpha(v,t)}&\les t^{-\frac14}\norm{z\chi(z)}_{L^2}t^{-\frac12}\norm{J_0(t)u(x,t)}_{L^2_x}\les t^{-\frac34}\norm{\jap{J_v(t)}u(t)}_{L^2},\\
  \norm{q(vt, t) - t^{-\frac12}\alpha(v,t)}_{L^2_v}&\les t^{-\frac12}\norm{z\chi(z)}_{L^1}t^{-\frac12}\norm{J_0(t)u(x,t)}_{L^2_x}\les t^{-1}\norm{\jap{J_v(t)}u(t)}_{L^2},
\end{aligned}  
\end{equation}
as claimed.
\end{proof}

Next we will isolate the well-known asymptotic dynamic for $\alpha$ (see e.g.\ \cite[Lemma 2.3]{IT}), provided that $u$ is a solution to \eqref{eq:mainPDE}. Due to the presence of the potential term, this requires a restriction to the {\em outer region}
\begin{equation}
    \Omega_t:=\{v:\abs{v}\geq 10 t^{-\frac12}\}.
\end{equation}
Moreover, the potential introduces a linear error term (cf.~\cite[Lemma 19]{Gavin}).

\begin{lem}\label{lem:alphaasymptotics} 
Assume that $u$ solves \eqref{eq:mainPDE}. Then we have that
\begin{equation}\label{eq:asymptoticODE}
\partial_t\alpha(v,t) = \mp \frac{i}{t}|\alpha(v,t)|^2\alpha(v,t)+R(v,t),\qquad R(v,t)=R_1(v,t)+R_2(v,t), 
\end{equation}
where the remainder terms $R_1,R_2$ satisfy the following bounds:
\begin{equation}\label{eq:R1_bd}
\begin{aligned}
  \norm{R_1(\cdot,t)}_{L^\infty}&\les t^{-\frac54}\big\|\langle J_V(t)\rangle u(\cdot,t)\big\|_{L^2}(1+\|t^{\frac12}u(\cdot,t)\|_{L^\infty}^2),\\
  \norm{R_1(\cdot,t)}_{L^2}&\les t^{-\frac32}\big\|\langle J_V(t)\rangle u(\cdot,t)\big\|_{L^2}(1+\|t^{\frac12}u(\cdot,t)\|_{L^\infty}^2),\\
\end{aligned}    
\end{equation}
and for any $\delta_1\in(0,1)$, there exists $C=C(\delta_1)>0$ such that
\begin{equation}\label{eq:R2_bd}
\begin{aligned}
\big\|R_2(\cdot,t)\big\|_{L^\infty(\Omega_t)}&\les t^{-\frac87}\big( C(\delta_1)\big\|\langle J_V(t)\rangle u(\cdot,t)\big\|_{L^2} +  \delta_1 \|t^{\frac12}u(\cdot,t)\|_{L^\infty} \big),\\
\big\|R_2(\cdot,t)\big\|_{L^2(\Omega_t)}&\les t^{-\frac54}\|t^{\frac12}u(\cdot,t)\|_{L^\infty}.
\end{aligned}
\end{equation}
\end{lem}

\begin{proof} As in \cite{IT} and Lemma~19 of~\cite{Gavin},  the key here is the identity
\begin{align*}
\partial_{t}\alpha(v, t) &= \mp i\int |u|^2 u(x,t)\overline{\Psi_v(x,t)}\,dx + \int u(x,t)\overline{(\partial_t + i\triangle)\Psi_v}\,dx+ i\int\frac{2u(x,t)}{1+x^2}\overline{\Psi_v}\,dx. 
\end{align*}
The leading order dynamic in \eqref{eq:asymptoticODE} comes from the first term. The last term gives the error term $R_2$, while the errors in the first and second terms are subsumed in $R_1$.

\medskip

{\it{(1) The estimates for $i\int\frac{2u(x,t)}{1+x^2}\overline{\Psi_v}\,dx$}}.  
We split this integral into two parts, with some $\gamma\in(1/2,1)$ that will be determined below:
\begin{align*}
i\int\frac{2u(x,t)}{1+x^2}\overline{\Psi_v}\,dx = i\int_{|x|\leq\delta_1^{-\frac12} t^\gamma}\frac{2u(x,t)}{1+x^2}\overline{\Psi_v}\,dx + i\int_{|x|>\delta_1^{-\frac12}t^\gamma}\frac{2u(x,t)}{1+x^2}\overline{\Psi_v}\,dx.
\end{align*}
Using Lemma~\ref{lem:Lem12}, as well as the fact that $|x|\gtrsim t^{\frac12}$ on the support of the integrand, provided that $\abs{v}\geq 100 t^{-\frac12}$, we find that for some $C=C(\delta_1)>0$  
\begin{align*}
\big|i\int_{|x|\leq \delta_1^{-\frac12}t^\gamma}\frac{2u(x,t)}{1+x^2}\overline{\Psi_v}\,dx\big| &\leq \|u(\cdot,t)\|_{L^\infty(|x|\le \delta_1^{-\frac12} t^\gamma)}  \int_{\sqrt{t} \les |x|} \frac{|\Psi_v(x,t)|}{1+x^2}\, dx\\
&\leq C(\delta_1) t^{-\frac32(1-\gamma)-\frac12}\big\|J_V (t)u(t)\big\|_{L^2}.    
\end{align*}
For the remaining integral, we obtain that 
\begin{align*}
\big|i\int_{|x|>\delta_1^{-\frac12}t^\gamma}\frac{2u(x,t)}{1+x^2}\overline{\Psi_v}\,dx\big|\leq   \delta_1 t^{-2\gamma}\, \big\|t^{\frac12}u(t)\big\|_{L^\infty}.
\end{align*}
The optimal choice is $\gamma=\frac47$, which gives a decay rate of $t^{-\frac87}$.

In order to derive the $L^2$-bound, we combine Minkowski's inequality with the estimate 
\begin{align*}
\norm{\Psi_v(x,t)}_{L^\infty_xL^2_v}\les t^{-\frac14}.
\end{align*}    
Taking advantage again of the fact that $|x|\gtrsim t^{\frac12}$ on the support of the integrand due to our assumption on $v$, it follows that
\begin{align*}
\Big\|i\int\frac{2u(x,t)}{1+x^2}\overline{\Psi_v}\,dx\Big\|_{L^2_v(R_t)}&\les \norm{\Psi_v(x,t)}_{L^\infty_xL^2_v}\int_{t^{\frac12}\les|x|}\frac{2|u(x,t)|}{1+x^2}\,dx\les t^{-\frac54}\big\|t^{\frac12}u(t)\big\|_{L^\infty}.     
\end{align*}
\smallskip
{\it{(2) The estimate for $\int u(x,t)\overline{(\partial_t + i\triangle)\Psi_v}\,dx$.}} As in \cite{IT} one has the identity
\begin{align*}
(i\partial_t -\triangle)\Psi_v = \frac{e^{-i\frac{x^2}{4t}}}{2t}\partial_x\Big(t^{\frac12}\chi'\big(\frac{x-vt}{\sqrt{t}}\big) + i(x-vt)\chi\big(\frac{x-vt}{\sqrt{t}}\big)\Big)    
\end{align*}
Using integration by parts, we deduce that 
\begin{align*}
 \int u(x,t)\overline{(\partial_t + i\triangle)\Psi_v}\,dx = -\frac{i}{2t}\int\partial_x\big(e^{i\frac{x^2}{4t}} u\big)\,\tilde{\chi}\,dx,  
\end{align*}
where 
\[
\tilde{\chi} = t^{\frac12}\left(\chi'\big(\frac{x-vt}{\sqrt{t}}\big) - i\frac{x-vt}{\sqrt{t}}\chi\big(\frac{x-vt}{\sqrt{t}}\big)\right).
\]
Since $\partial_x\big(e^{i\frac{x^2}{4t}} u\big) = \frac{i}{2t}e^{i\frac{x^2}{4t}}J_0(t)u$, we can use the Cauchy-Schwarz inequality as well as Corollary~\ref{cor:J_0J_Vbound} to estimate 
\begin{align*}
\Big|-\frac{i}{2t}\int\partial_x\big(e^{i\frac{x^2}{4t}} u\big)\,\tilde{\chi}\,dx\Big|\lesssim |t|^{-2}\big\|\langle J_V(t) \rangle u(t)\big\|_{L^2}\, \big\|\tilde{\chi}\big\|_{L_x^2}\lesssim |t|^{-\frac54}\big\|\langle J_V(t) \rangle u(t)\big\|_{L^2}.
\end{align*}
For the $L^2_v$-bound, we observe that writing $\tilde{x} = x-vt$, and setting $\tilde{\chi} = \tilde{\chi}(\tilde{x}, t)$, we have 
\begin{align*}
\big\|\tilde{\chi}(\tilde{x}, t)\big\|_{L^1_{\tilde{x}}}\les t.  
\end{align*}
We conclude that 
\begin{align*}
\Big\|-\frac{i}{2t}\int\partial_x\big(e^{i\frac{x^2}{4t}} u\big)\,\tilde{\chi}\,dx\Big\|_{L^2_v} &= \Big\|-\frac{i}{2t}\int\partial_x\big(e^{i\frac{x^2}{4t}} u\big)(\tilde{x}+ vt)\,\tilde{\chi}(\tilde{x},t)\,d\tilde{x}\Big\|_{L^2_v}\les t^{-\frac32}\big\|J_V(t)u(t)\big\|_{L^2}. 
\end{align*}
\smallskip
{\it{(3) The estimates for $\mp i\int |u|^2 u(x,t)\overline{\Psi_v(x,t)}\,dx$.}} As in \cites{IT,Gavin} we write this term as the sum of a principal contribution and two error terms:
\begin{align*}
i\int |u|^2 u(x,t)\overline{\Psi_v(x,t)}\,dx &= \frac{i}{t}\alpha(v,t)|\alpha(v,t)|^2 + i\alpha(v,t)\big(|u(vt,t)|^2 - t^{-1}|\alpha(v,t)|^2\big)\\
&\qquad + i\int\big(|u(x,t)|^2 - |u(vt,t)|^2\big)u(x,t)\overline{\Psi_v(x,t)}\,dx\\
& =: \frac{i}{t}\alpha(v,t)|\alpha(v,t)|^2 + R_{1,1} + R_{1,2}. 
\end{align*}
In order to estimate these error terms, we invoke the bounds from the preceding lemma:

\smallskip

{\it{(a) The estimate for $R_{1,1}$.}} We can bound 
\begin{align*}
 \Big|i\alpha(v,t)\big(|u(vt,t)|^2 - t^{-1}|\alpha(v,t)|^2\big)\Big|&\lesssim \big|\alpha(v,t)\big|\, \big|u(vt,t) - t^{-\frac12}e^{-i\frac{x^2}{4t}}\alpha(v,t)\big|\big(\big|u(vt,t)\big| + t^{-\frac12}\big|\alpha(v,t)\big|\big)\\
&\lesssim t^{-\frac54}\big\|\langle J_V(t)\rangle u(t)\big\|_{L^2}\, \big\|t^{\frac12}u(t)\big\|_{L^\infty}^2. 
\end{align*}
For the $L^2_v$-norm we obtain 
\begin{align*}
&\Big\|i\alpha(v,t)\big(|u(vt,t)|^2 - t^{-1}|\alpha(v,t)|^2\big)\Big\|_{L^2_v}\\
&\qquad \lesssim \big\|u(vt,t) - t^{-\frac12}e^{-i\frac{x^2}{4t}}\alpha(v,t)\big\|_{L^2_v}\big\|\alpha(v,t)\big\|_{L^\infty_v}\big(\big\|u(vt,t)\big\|_{L^\infty_v} + t^{-\frac12}\big\|\alpha(v,t)\big\|_{L^\infty_v}\big)\\
&\qquad \lesssim t^{-\frac32}\big\|\langle J_V(t)\rangle u(\cdot,t)\big\|_{L^2}\big\|t^{\frac12}u(t)\big\|_{L^\infty}^2. 
\end{align*}

{\it{(b) The estimate for $R_{1,2}$.}} Again invoking $q = e^{i\frac{x^2}{4t}}u$, we estimate
\begin{align*}
|u|\Big||u(x,t)|^2 - |u(vt,t)|^2\Big|\lesssim t^{-1}\big\|t^{\frac12}u\big\|_{L_x^\infty}^2\,\big|q(x, t) - q(vt, t)\big|.      
\end{align*}
Letting instead $x = (v-z)t$, and using \eqref{eq:qdiff-expand} to bound
\begin{equation*}
 \abs{R_{1,2}(v,t)}\les \big\|t^{\frac12}u(t)\big\|_{L^\infty}^2 \int_0^1\int \abs{t\p_xq\big((v-hz)t, t\big)}z\chi(t^{\frac12}z) dzdh,
\end{equation*}
as in the proof of Lemma~\ref{lem:mainprofile} we infer that
\begin{equation*}
  \abs{R_{1,2}(v,t)}  \les \big\|t^{\frac12}u(t)\big\|_{L^\infty}^2 t^{-\frac12}\big\|\langle J_V(t)\rangle u(t)\big\|_{L^2}t^{-\frac34}\lesssim t^{-\frac54}\, \big\|t^{\frac12}u(t)\big\|_{L^\infty}^2 \big\|\langle J_V(t)\rangle u(t)\big\|_{L^2}.
\end{equation*}
For the $L^2_v$-bound, we deduce that 
\begin{equation*}
  \norm{R_{1,2}(v,t)}_{L^2_v}  \les \big\|t^{\frac12}u(t)\big\|_{L^\infty}^2 t^{-\frac12}\big\|\langle J_V(t)\rangle u(t)\big\|_{L^2}t^{-1}\lesssim t^{-\frac32}\, \big\|t^{\frac12}u(t)\big\|_{L^\infty}^2 \big\|\langle J_V(t)\rangle u(t)\big\|_{L^2}.
\end{equation*}
This completes the proof of the lemma. 
\end{proof}

We immediately deduce the following result, which will be useful for our bootstrap argument in the proof of Theorem \ref{thm:main}.
\begin{cor}\label{cor:aprioribound}
There exists a universal constants $M>0$ such that the following holds. If $u$ is a solution to \eqref{eq:mainPDE} on $[1,T)$, $1\leq T\leq \infty$, satisfying for some small $0<\delta< \frac14$ that
\begin{equation}
    A(T):=\sup_{t\in [1,T)}t^{\frac12}\|u(\cdot, t)\|_{L^\infty}<\infty,\qquad B(T):=\sup_{t\in [1,T)}t^{-\delta}\|\langle J_V(t)\rangle u(\cdot, t)\|_{L^2}<\infty,
\end{equation}
then there holds that
\begin{equation}
  A(T)\leq M B(T)(1+A(T)^2).  
\end{equation}
\end{cor}
\begin{proof}
  On the one hand, by Lemma~\ref{lem:Lem12} there exists $C_1>0$ such that for all $t\in [1,T]$
  \begin{equation}
      \norm{u(t)}_{L^\infty(\abs{x}\leq 100 t^{\frac12})}\leq C_1t^{-\frac34}\norm{J_V(t)u(t)}_{L^2}\leq C_1t^{-\frac34+\delta}B(T).
  \end{equation}
  On the other hand, if $\abs{x}>100 t^{\frac12}$, then we note that $v:=\frac{x}{t}\in\Omega_t$, and thus by Lemma \ref{lem:mainprofile} there exists $C_2>0$ such that
  \begin{equation}
      \abs{t^{\frac12}u(x,t)-e^{-i\frac{x^2}{4t}}\alpha(\frac{x}{t},t)}\leq C_2\, t^{-\frac14+\delta}B(T).
  \end{equation}
  It thus remains to bound $\alpha(v,t)$ for $v\in\Omega_t$. For a given such $v\in\Omega_t$, we let $t_0(v):=\max\{1,100\abs{v}^{-2}\}$. We infer from $|vt_0^{\frac12}(v)|\leq 10$ that $\textnormal{supp}\,{\Psi_v(\cdot,t_0)}\subset (-50t_0^{\frac12},50t_0^{\frac12})$. It follows with Lemma \ref{lem:Lem12} and $\norm{\Psi_v(\cdot,t_0)}_{L^1}=t_0^{\frac12}$ that
  \begin{equation}\label{eq:alphavt0v}
    \abs{\alpha(v,t_0(v))}\leq \norm{u(t_0)}_{L^\infty(\abs{x}\leq 50t_0^{\frac12})}\norm{\Psi_v(\cdot,t_0)}_{L^1}\leq C_1t_0^{-\frac14+\delta}B(T).
  \end{equation}
  Finally, for $t\geq t_0(v)$ we have that $v\in\Omega_t$, and we can thus appeal to Lemma \ref{lem:alphaasymptotics}. Thus, integrating the ODE \eqref{eq:asymptoticODE}, we have that
  \begin{equation}
      \frac{\partial}{\partial t}\big(e^{\pm i\int_1^t s^{-1}|\alpha(v, s)|^2\,ds}\, \alpha(v, t)\big) = e^{\pm i\int_1^t s^{-1}|\alpha(v, s)|^2\,ds}R(v, t),
  \end{equation}
  and thus
  \begin{equation}
     \big|\alpha(v, t)\big|\leq\big|\alpha(v, t_0(v))\big| + \int_{t_0}^t \big|R(v, s)\big|\,ds. 
  \end{equation}
  Integrating the bounds \eqref{eq:R1_bd}, \eqref{eq:R2_bd} for $R$ from Lemma \ref{lem:alphaasymptotics} and choosing $\delta_1=\frac12$ yields
  \begin{equation}
  A(T)\leq M B(T)(1+A(T)^2)+\frac12 A(T),  
\end{equation}
where $M>0$ is a universal constant. 
\end{proof}

\section{Global existence and modified scattering}\label{sec:gescat}

We can now prove the main result of this paper, which is a more detailed version of Theorem~\ref{thm:main}.
\begin{thm}\label{thm:main2}
There exists $\eps_0>0$ such that the following holds. Given $u_*\in H^1(\mathbb{R})$ satisfying
\begin{equation}
\big\|\langle x\rangle u_*\big\|_{L^2} + \big\|u_*\big\|_{H^1}\leq\eps\leq\eps_0,
\end{equation}
the Cauchy problem
\begin{equation}\label{eq:main_IVP}
 \begin{cases}  
  i\p_t u +\calL u =\mu |u|^2 u,\qquad \mu\in\{-1,+1\},\\
  u(1,x)=e^{i\calL}u_*(x),
 \end{cases}   
\end{equation}
has a unique global solution $u\in C([1,\infty),H^1(\R))$ satisfying for all $t\geq 1$ the bounds
\begin{equation}\label{eq:global_bds}
  \norm{u(t)}_{L^\infty}\les \eps t^{-\frac12},\quad \norm{\jap{J_V(t)}u(t)}_{L^2}\les \eps t^{C\eps^2},
\end{equation}
where $C>0$ is a universal constant.

Moreover, there exist a unique $u_\infty\in L^2(\R)\cap L^\infty(\R)$ and $\Phi_\infty\in L^\infty(\R)$ such that as $t\to\infty$
\begin{equation}\label{eq:mod_scat}
  u(x,t)=t^{-\frac12}e^{\left(-i\frac{x^2}{4t}-i\mu \abs{u_\infty\left(\frac{x}{t}\right)}^2\log(t)-i\mu\Phi_\infty\left(\frac{x}{t}\right)\right)}u_\infty\left(\frac{x}{t}\right)+R[u](x,t),  
\end{equation}
with $R[u]\in L^2(\R)\cap L^\infty(\R)$ satisfying
\begin{equation}
    \norm{R[u](t)}_{L^\infty}\les \eps t^{-\frac{3}{5}},\qquad \norm{R[u](t)}_{L^2}\les\eps t^{-\frac{1}{5}}.
\end{equation}
\end{thm}

\begin{proof}
 The local well-posedness of solutions to \eqref{eq:main_IVP} satisfying also that $\norm{\jap{J_V(t)}u(t)}_{L^2}<+\infty$ is standard (see also Lemma \ref{lem:crudeJvbound}), and it thus suffices to show a continuity argument for the norm control claimed in \eqref{eq:global_bds}. 
 
 To this end, with the notation of Corollary \ref{cor:aprioribound} and noting that $\norm{u(1)}_{L^\infty}+\norm{J_V(1)u(1)}_{L^2}\les \norm{\jap{x}u_*}_{H^1}\les\eps$, assume that for some $T>1$ we are given a solution to \eqref{eq:main_IVP} on $[1,T]$ satisfying for some absolute constant $C_*>0$ (to be determined below) and $\delta>0$ that
 \begin{equation}
   A(T)=\sup_{t\in [1,T)}t^{\frac12}\|u(\cdot, t)\|_{L^\infty}\leq M\eps^{\frac12},\qquad B(T)=\sup_{t\in [1,T)}t^{-\delta}\|\langle J_V(t)\rangle u(\cdot, t)\|_{L^2}\leq \eps^{\frac12},  
 \end{equation}
 where $M>0$ is as in Corollary \ref{cor:aprioribound}. Then from the key commutation identity for $J_V(t)$,
 \begin{equation}
   J_V(t)e^{i(t-s)\calL}=e^{i(t-s)\calL}J_V(s),
 \end{equation}
 and Duhamel's formula for solutions to \eqref{eq:main_IVP}
 \begin{equation*}
    u(t) = e^{it\mathcal{L}}u_* -i\mu\int_1^t e^{i(t-s)\mathcal{L}}\abs{u}^2u(s)\,ds
 \end{equation*}
 it follows with Proposition \ref{prop:cubic} that
 \begin{equation}
 \begin{aligned}
     \norm{J_V(t)u(t)}_{L^2}&\leq \norm{J_V(t)e^{it\mathcal{L}}u_*}_{L^2}+\int_1^t\norm{J_V(s)\abs{u}^2u(s)}_{L^2}ds\\
     &\leq \norm{J_V(1)u(1)}+C(A(T)+B(T))^2\int_1^ts^{-1}\norm{\jap{J_V(s)} u(s)}_{L^2}ds\\
     &\leq C\eps+C\eps^2\int_1^ts^{-1}\norm{\jap{J_V(s)} u(s)}_{L^2}ds,
 \end{aligned}    
 \end{equation}
 where $C>0$ is a universal constant that may change from line to line.
 Together with the mass conservation of \eqref{eq:main_IVP}, by Gr\"onwall's inequality we obtain that for all $t\in[1,T]$
 \begin{equation}
    \norm{\jap{J_V(t)}u(t)}_{L^2}\leq C_1 \eps t^{C_2\eps^2}\leq\frac14 \eps^{\frac12} t^\delta,
 \end{equation}
 provided that $\eps\leq \eps_0$ is sufficiently small. This closes the bootstrap for $B(T)$. From Corollary \ref{cor:aprioribound} it follows further that
 \begin{equation}
     t^{\frac12}\norm{u(t)}_{L^\infty}\leq M \frac14 \eps^{\frac12}(1+M^2\eps)\leq\frac12 M\eps^{\frac12},
 \end{equation}
 thereby also closing the bootstrap for $A(T)$.

 It remains to establish the asymptotic dynamic for $u$. Thanks to Lemma \ref{lem:mainprofile}, to leading order this is given by the dynamics of $\alpha$. In the outer region $\Omega_t$, by Lemma \ref{lem:alphaasymptotics} we have that
 \begin{equation}
     \beta(v,t):=e^{i\mu\int_1^t s^{-1}|\alpha(v, s)|^2\,ds}\, \alpha(v, t)
 \end{equation}
 satisfies
 \begin{equation}
     \p_t\beta(v,t)=e^{i\mu\int_1^t s^{-1}|\alpha(v, s)|^2\,ds}R(v,t)
 \end{equation}
 and thus for $1\leq t_1<t_2<\infty$
 \begin{equation}
     \abs{\beta(v,t_2)-\beta(v,t_1)}\leq\int_{t_1}^{t_2}\abs{R(v,s)}ds.
 \end{equation}
 Hence if $v\in\Omega_{t_0}$ for some $t_0\geq 1$, the bounds \eqref{eq:R1_bd} and \eqref{eq:R2_bd} imply that $\beta$ is a Cauchy sequence in $L^\infty(\Omega_{t_0})\cap L^2(\Omega_{t_0})$, and we let $u_\infty(v):=\lim_{t\to\infty}\beta(v,t)\in L^\infty(\R)\cap L^2(\R)$. The same error bounds and \eqref{eq:global_bds} then show that
 \begin{equation}\label{eq:uinfty_errors}
     \abs{u_\infty(v)-\beta(v,t)}\les \eps t^{-\frac{1}{10}},\quad \norm{u_\infty(v)-\beta(v,t)}_{L^2(\Omega_t)}\les \eps t^{-\frac{1}{5}},
 \end{equation}
 which gives that
 \begin{equation}
 \begin{aligned}
   \one_{\abs{x}>10t^{\frac12}}\abs{u(x,t)-t^{-\frac12}e^{\left(-i\frac{x^2}{4t}-i\mu \int_1^t s^{-1}\abs{\alpha\left(\frac{x}{t},s\right)}^2ds\right)}u_\infty\left(\frac{x}{t}\right)}&\les \eps t^{-\frac12-\frac{1}{10}},  \\
   \norm{u(x,t)-t^{-\frac12}e^{\left(-i\frac{x^2}{4t}-i\mu \int_1^t s^{-1}\abs{\alpha\left(\frac{x}{t},s\right)}^2ds\right)}u_\infty\left(\frac{x}{t}\right)}_{L^2(\abs{x}>10t^{\frac12})}&\les \eps t^{-\frac15}. 
 \end{aligned}  
 \end{equation}
 On the other hand, for $0<\abs{v}\ll 1$ we use \eqref{eq:uinfty_errors} for the choice $t=t_0(v)=100\abs{v}^{-2}$ and bound $\abs{\alpha(v,t_0(v)}\les \eps t^{-\frac14+C\eps^2}$ as in \eqref{eq:alphavt0v} to obtain that
 \begin{equation}
   \abs{u_\infty(v)}\les \abs{v}^{\frac15}.  
 \end{equation}
 This implies that for $\abs{x}\leq 10 t^{\frac12}$ we have $\abs{u_\infty\left(\frac{x}{t}\right)}\les t^{-\frac{1}{10}}$. With Lemma \ref{lem:Lem12} and the bounds \eqref{eq:global_bds} it then follows that 
 \begin{equation}
 \begin{aligned}
    \norm{u(t)}_{L^\infty(\abs{x}\leq 10t^{\frac12})}+t^{-\frac12}\norm{u_\infty\left(\frac{x}{t}\right)}_{L^\infty(\abs{x}\leq 10t^{\frac12})}&\les \eps t^{-\frac34+C\eps^2}+ \eps t^{-\frac12-\frac{1}{10}},\\
    \norm{u(t)}_{L^2(\abs{x}\leq 10t^{\frac12})}+t^{-\frac12}\norm{u_\infty\left(\frac{x}{t}\right)}_{L^2(\abs{x}\leq 10t^{\frac12})}&\les \eps t^{-\frac34+C\eps^2}t^{\frac14}+\eps t^{-\frac{7}{20}}.
 \end{aligned}   
 \end{equation}
 Altogether we have shown that
 \begin{equation}
 \begin{aligned}
     \norm{u(x,t)-t^{-\frac12}e^{\left(-i\frac{x^2}{4t}-i\mu \int_1^t s^{-1}\abs{\alpha\left(\frac{x}{t},s\right)}^2ds\right)}u_\infty\left(\frac{x}{t}\right)}_{L^\infty}\les \eps t^{-\frac35},\\
     \norm{u(x,t)-t^{-\frac12}e^{\left(-i\frac{x^2}{4t}-i\mu \int_1^t s^{-1}\abs{\alpha\left(\frac{x}{t},s\right)}^2ds\right)}u_\infty\left(\frac{x}{t}\right)}_{L^2}\les \eps t^{-\frac15}.
 \end{aligned}    
 \end{equation}
 Finally, using that in the exponential also $\Phi(t):=\int_1^ts^{-1}\abs{\alpha\left(v,s\right)}^2ds-\log(t)\abs{\alpha\left(v,t\right)}^2$ are Cauchy as $t\to\infty$ and $\abs{\alpha(v,t)}=\abs{\beta(v,t)}$, we obtain the claimed asymptotics (see e.g.\ \cite[pages 381-383]{HayNau}).
\end{proof}

\newpage

\bibliographystyle{amsplain}
\bibliography{refs}

 \Addresses

\end{document}